\documentclass[a4paper,11pt,english]{amsart}
\usepackage{babel}
\usepackage{tikz}
\usepackage{hyperref}
\usepackage{here}

\usetikzlibrary{arrows,decorations.markings,positioning}
\tikzset{inner sep=0pt, node distance=5mm,
  root/.style={circle,draw,minimum size=5pt,thick},
  broot/.style={circle,draw,minimum size=5pt,thick,fill},
  xroot/.style={circle,draw,minimum size=5pt,thick,label=below:$\times$},
  doublearrow/.style={postaction={decorate},   decoration={markings,mark=at position .6 with {\arrow[line width=1.2pt]{>}}},double distance=1.6pt,thick},
  rdoublearrow/.style={postaction={decorate},   decoration={markings,mark=at position .4 with {\arrowreversed[line width=1.2pt]{>}}},double distance=1.6pt,thick},
  curvedline/.style={bend=right}
}

\newtheorem{theorem}{Theorem}[section]
\newtheorem{lemma}[theorem]{Lemma}
\newtheorem{proposition}[theorem]{Proposition}

\theoremstyle{definition}
\newtheorem{definition}[theorem]{Definition}

\theoremstyle{remark}
\newtheorem{remark}[theorem]{Remark}

\numberwithin{equation}{section}
\newcommand{\R}{\mathbb{R}}
\newcommand{\C}{\mathbb{C}}
\newcommand{\K}{\mathbb{K}}
\newcommand{\Z}{\mathbb{Z}}
\newcommand{\Cl}{C\!\ell}
\newcommand{\Id}{\mathrm{Id}}
\newcommand{\so}{\mathfrak{so}}
\newcommand{\g}{\mathfrak{g}}
\newcommand{\ra}{\mathfrak{r}}
\newcommand{\h}{\mathfrak{h}}
\newcommand{\p}{\mathfrak{p}}
\renewcommand{\c}{\mathfrak{c}}
\renewcommand{\a}{\mathfrak{a}}
\renewcommand{\t}{\mathfrak{t}}

\newcommand{\m}{\mathfrak{m}}

\newcommand{\ad}{\mathrm{ad}}
\newcommand{\Lie}{\mathrm{Lie}}
\newcommand{\vol}{\mathrm{vol}}
\newcommand{\sign}{\mathrm{sign}}

\begin{document}
\title[{Tanaka structures modeled on extended Poincar\'e algebras}]
{Tanaka structures modeled on \\ extended Poincar\'e algebras}
\author[A.~Altomani]{Andrea Altomani}
\address[A.~Altomani]{Research Unit in Mathematics\\
Universit\'e du Luxembourg\\
6 rue Richard Coudenhove-Kalergi\\
L-1359 Luxembourg.
}
\email{andrea.altomani@uni.lu}
\author[A.~Santi]{Andrea Santi}
\address[A.~Santi]{Research Unit in Mathematics\\
Universit\'e du Luxembourg\\
6 rue Richard Coudenhove-Kalergi\\
L-1359 Luxembourg.
}\email{andrea.santi@uni.lu}
\thanks{\\ \phantom{ccc}The second author was supported by the project F1R-MTH-PUL-08HALO-HALOS08 of the University of Luxembourg.}

\keywords{Tanaka prolongations, extended Poincar\'e algebras, extended translation algebras, Clifford algebras and spinors}
\subjclass[2000]{53C30, 58A30, 53C27, 53C80}

\begin{abstract} Let $(V,(\cdot,\cdot))$ be a pseudo-Euclidean vector space and $S$ an irreducible $\Cl(V)$-module. An extended translation algebra is a graded Lie algebra
$\mathfrak{m}=\mathfrak{m}_{-2}+\mathfrak{m}_{-1}=V+S$ with bracket given by $([s,t],v)=b(v\cdot s,t)$ for some nondegenerate $\mathfrak{so}(V)$-invariant reflexive bilinear form $b$ on $S$.
An extended Poincar\'e structure on a manifold $M$ is a regular distribution $\mathcal{D}$ of depth $2$ whose Levi form $\mathcal{L}_x:\mathcal{D}_x\wedge\mathcal{D}_x\rightarrow T_x M/\mathcal{D}_x$ at any point $x\in M$ is identifiable with the bracket $[\cdot,\cdot]\colon S\wedge S\rightarrow V$ of a fixed extended translation algebra $\mathfrak{m}$. 
The classification of the standard maximally homogeneous manifolds with an extended Poincar\'e structure is given, in terms of Tanaka prolongations of extended translation algebras and of appropriate gradations of real simple Lie algebras.
\end{abstract}
\maketitle
\section{Introduction}
Let $(V,(\cdot,\cdot))$ be a pseudo-Euclidean vector space and denote by $S$ an irreducible module for the associated Clifford algebra $\Cl(V)$. An {\it extended Poincar\'e algebra} is a graded Lie algebra 
$\mathfrak{p}=\mathfrak{p}_{-2}+\mathfrak{p}_{-1}+\mathfrak{p}_{0}=V+S+\mathfrak{so}(V)$ such that the
adjoint actions $\ad(\so(V))|_V$ and $\ad(\so(V))|_S$ are, respectively, the defining and spin representations of 
the orthogonal Lie algebra 
$\so(V)$.

In \cite{AC} it was observed that any extended Poincar\'e algebra is completely determined by the $\so(V)$-invariant tensor
\begin{equation}
\label{bracketpoinc}
\Gamma=[\cdot,\cdot]|_{S\otimes S}\in \Lambda^2 S^*\otimes V
\end{equation}
that defines the Lie bracket between elements of $S$.
Such a tensor \eqref{bracketpoinc}, and the associated extended Poincar\'e algebra, were called {\it admissible} if $(\Gamma(s,t),v)=b(v\cdot s,t)$ for some nondegenerate $\so(V)$-invariant reflexive bilinear form $b$ on $S$ (whenever $S=S^++S^-$ decomposes as an $\so(V)$-module, it was also required that $S^+$ and $S^-$ were isotropic or mutually orthogonal w.r.t. $b$, see subsections \ref{subsub1}, \ref{subsub2}). 

An analogous picture was given in the case of graded Lie superalgebras: a {\it super Poincar\'e algebra} $\mathfrak{p}=\mathfrak{p}_{-2}+\mathfrak{p}_{-1}+\mathfrak{p}_{0}=V+S+\mathfrak{so}(V)$ is completely determined by a symmetric $\so(V)$-invariant tensor $\Gamma\in\vee^2 S^*\otimes V$. 

The main result of \cite{AC} is an explicit description of a basis, consisting of admissible tensors, of the space of $\so(V)$-invariant elements in $S^*\otimes S^*\otimes V$.

Admissible extended Poincar\'e algebras $\p$ in signature $\sign(V)=(3,3+k)$ are closely related to the isometry algebra of a simply connected homogeneous quaternionic
K\"ahler manifold $N=L/K$ \cite{C}. Indeed, the Lie algebra $\Lie(L)$ is obtained by adding to 
$\mathfrak{p}$ the outer derivation $E$ which satisfies $\ad(E)|_{\p_i}=i\mathrm{Id}_{\p_i}$ for every $-2\leq i\leq 0$; the stabilizer $K$ is a maximal connected compact subgroup of $L$.
One can then see \cite{AC2,dWVVP} that the full isometry algebra $\mathfrak{iso}(N)$ is the semi-direct sum of $\Lie(L)$ and a compact subalgebra of
$
\h_0=\{D\in\operatorname{Der}_0(\p)\mid Dv=0\ \forall v\in V\}
$. 
The algebras $\mathfrak{iso}(N)$ so constructed are subalgebras of non-positive degree of the {\it maximal transitive prolongation} (see section \ref{sec1} or \cite{N2} for the definition) of the two-step nilpotent Lie algebra $\m=V+W$. 

\smallskip

To a lesser extent, extended Poincar\'e algebras have also appeared in the Physics literature, e.g. in relation to self-dual Yang-Mills theory \cite{DL}.

On the other hand, super Poincar\'e algebras are a central object in the construction of supersymmetric field theories; we refer to
the books \cite{DP, F} as a starting point for the vast literature on this topic.
In \cite{SaS1,SaS2} it was observed that, whenever the extended (resp. super) Poincar\'e algebra $\mathfrak{p}=\mathfrak{p}_{-2}+\mathfrak{p}_{-1}+\mathfrak{p}_{0}$ is admissible, any connected homogeneous manifold (resp. supermanifold) $M=G/Q$ with $\Lie(G)=\mathfrak{p}$ and $\Lie(Q)=\mathfrak{p}_0$ is endowed with a $G$-invariant distribution $\mathcal{D}$ of depth $2$ (i.e. $\mathcal{D}+[\mathcal{D},\mathcal{D}]$ spans the tangent space at every point) such that $\mathcal{D}_o\simeq\mathfrak{p}_{-1}=S$ at the origin $o=eQ$. Its Levi form at $x\in M$
\[
\mathcal{L}_x:\mathcal{D}_x\wedge\mathcal{D}_x\rightarrow T_x M/\mathcal{D}_x\ ,\quad (X,Y)\mapsto [X,Y]|_x \mod \mathcal{D}_x\,,
\]
is, up to linear isomorphism, independent of $x$ and indentifiable with the tensor \eqref{bracketpoinc}.
Homogeneous supermanifolds $M=G/Q$ of this kind can be considered as the vacua models for those supersymmetric field theories known by the name of supergravity theories
(see e.g. \cite{L,RKS, RS1, RS2, SaS1, SaS2} and references therein).

\smallskip

We define, inspired by \cite{SaS1, SaS2}, a geometrical structure on a manifold that can be considered as the curved analogue of the above homogeneous space.
\begin{definition}
\label{struttura}
Let $\mathfrak{p}=\mathfrak{p}_{-2}+\mathfrak{p}_{-1}+\mathfrak{p}_0=V+S+\mathfrak{so}(V)$ be an extended Poincar\'e algebra determined by the tensor \eqref{bracketpoinc} associated to a nondegenerate $\so(V)$-invariant reflexive bilinear form $b$ on $S$.
An {\it extended Poincar\'e structure} on a manifold $M$ of dimension $\dim M=\dim V+\dim S$ is a depth $2$ smooth distribution $\mathcal{D}$ with rank $\mathcal{D}=\dim S$, and Levi form $\mathcal{L}_x:\mathcal{D}_x\wedge\mathcal{D}_x\rightarrow T_x M/\mathcal{D}_x$ identifiable at any point $x\in M$ with the tensor \eqref{bracketpoinc} determining $\mathfrak{p}=V+S+\so(V)$.
\end{definition}
A similar definition of {\it super Poincar\'e structure} can be given for a supermanifold $M$ endowed with an odd distribution $\mathcal{D}$.
Here, we restrict ourselves to the non-super case. This has a double motivation: on one hand, the theory of Lie algebras provides useful insights 
into the more challenging Lie superalgebra case; on the other hand, we believe that extended Poincar\'e structures have their own interest from a purely mathematical point of view.
The ``super'' case will be discussed in the upcoming paper \cite{AS2}.
\vskip0.5cm
The main aim of this work is to initiate the investigation of the geometry of manifolds endowed with an extended Poincar\'e structure. 
We call {\it extended translation algebra} the negatively graded part
\begin{equation}
\label{negpart}
\mathfrak{m}=\mathfrak{m}_{-2}+\mathfrak{m}_{-1}=\mathfrak{p}_{-2}+\mathfrak{p}_{-1}=V+S
\end{equation} 
of an extended Poincar\'e algebra $\mathfrak{p}$ as in Definition \ref{struttura} and denote by
\begin{equation}
\label{mtp}
\mathfrak{g}=\mathfrak{m}_{-2}+\mathfrak{m}_{-1}+\mathfrak{g}_{0}+\mathfrak{g}_{1}+\cdots
\end{equation}
its maximal transitive prolongation in the sense of Tanaka.

\smallskip

If $\dim V\geq 3$, we prove that \eqref{mtp} is finite-dimensional. In this case, any manifold $(M,\mathcal{D})$ with an extended Poincar\'e structure is in particular endowed with a Tanaka structure (see e.g. \cite{AS} for the definition) of finite type. Classical Tanaka theory \cite{N1, N2} implies that the automorphism group $\mathrm{Aut}(M,\mathcal{D})$ is a Lie group with dimension bounded from above by $\dim \mathfrak{g}$. 

The quotient $M=G/Q$ of the simply connected Lie group $G$ with Lie algebra $\Lie(G)=\g$,
by the connected subgroup $Q$ with Lie algebra $\mathrm{Lie}(Q)=\sum_{i\geq 0} \mathfrak{g}_{i}$, has a canonical extended Poincar\'e structure whose group of automorphisms
is locally isomorphic to $G$; in the terminology of Tanaka, $M$ is the {\it maximally homogeneous model}.
Our main result is the classification of maximal transitive prolongations of extended translation algebras (Theorems \ref{abcdefg} and \ref{N1r}) 
and, consequently, the classification of maximally homogeneous models. 

For all but a finite number of signatures of $V$,
the maximal transitive prolongation \eqref{mtp} is the semi-direct sum of $\m$ with $\g_0=\so(V)+\R E+\h_0$, where $E$ is the derivation of $\m$ satisfying $\ad(E)|_{\m_i}=i\mathrm{Id}_{\m_i}$ for $-2\leq i\leq -1$ and $\h_0=\{D\in\operatorname{Der}_0(\m)\mid Dv=0\ \forall v\in V\}$. 
In the remaining cases, $\g=\g_{-2}+\dots+\g_{2}$ is a simple Lie algebra of type $\mathrm{A}$, $\mathrm{C}$, $\mathrm{E}_6$ or $\mathrm{F} _4$ as listed in Theorem \ref{N1r}. 

In the case of the exceptional Lie algebras $\mathrm{E}_6$ and $\mathrm{F} _4$, we remark that our construction is similar in spirit, but different, to the one given in \cite{A}.
Note also that when $\m$ is the extended translation algebra associated to a homogeneous quaternionic K\"ahler manifold $N=L/K$, the additional symmetries arising from elements of positive degree are not isometries of $N$; it might be interesting to understand their geometrical significance.

Whenever \eqref{mtp} is simple, an extended Poincar\'e structure is a parabolic geometry in the sense of e.g. \cite{AS, CS} and the associated flag manifolds are precisely the maximally homogeneous models.  We point out that the action of $\Cl(V)$ on $S$ gives rise to a canonical family of partial almost \mbox{(para-)}complex structures, for which the linear connections described in \cite{SaS1} satisfy compatibility conditions similar to those of the Tanaka-Webster connection (\cite{Tc, W}, see \cite{DT} for further references) in pseudo-Hermitian geometry. 
Deeper investigation on this subject will be the content of future works. 

\smallskip

If $\dim V\leq 2$, the maximal transitive prolongation \eqref{mtp} is a real or complex contact Lie algebra. In both cases, it is infinite-dimensional and the maximally homogeneous model is properly defined only locally. This reflects the fact that the extended Poincar\'e structure reduces to a contact structure. 

\smallskip

Let us finally note that more general extended Poincar\'e algebras $\mathfrak{p}=\mathfrak{p}_{-2}+\mathfrak{p}_{-1}+\mathfrak{p}_{0}=V+W+\so(V)$ 
have also been considered, where $W=S\oplus\cdots\oplus S$ is the direct sum of $N$ copies of an irreducible $\Cl(V)$-module $S$ (see e.g. \cite{AC}).
If $N>1$, Theorem \ref{abcdefg} classifies the maximal transitive prolongations
of extended translation algebras $\mathfrak{m}=\mathfrak{m}_{-2}+\mathfrak{m}_{-1}=V+W$ which are simple. 

\vskip0.5cm

The paper is structured as follows. 
In Section \ref{sec1} we give the relevant definitions, compute explicitly the $\g_0$ component of the maximal transitive prolongation $\g$ (Theorem \ref{lo0}), and prove that for $\dim V\geq 3$ the Lie algebra $\g$ is finite dimensional (Theorem \ref{fine}). Moreover, for $N=1$, we show that either $\g_i=0$ for all $i>0$, or $\g$ is simple (Theorem \ref{nonloso}).

Section \ref{sec2} contains the classification of all simple Lie algebras $\g$ that arise as maximal transitive prolongations of extended translation algebras, for arbitrary $N\geq 1$ (Theorem \ref{theoC} and Theorem \ref{abcdefg}). The classification is given in terms of Dynkin and Satake diagrams with additional data describing the gradation $\g=\sum_i\g_i$.

In Section \ref{sec3} we describe the cases with $\dim V=1,2$, and show that $\g$ is an infinite dimensional Lie algebra of contact type.

The final Section \ref{sec4} contains the full classification for $N=1$ (Theorem \ref{N1c} and Theorem \ref{N1r}).

\medskip

\noindent{\it Acknowledgements}. The authors are grateful to A.\ Spiro for useful discussions and helpful comments.

\section{The Tanaka prolongation.}
\label{sec1}
Let $V$ be a finite dimensional vector space over $\K=\R$ or $\K=\C$, with a nondegenerate, possibly indefinite in the real case, symmetric bilinear form $(\cdot,\cdot)$, and let $\Cl(V)=\Cl(V)^{\overline{0}}+\Cl(V)^{\overline{1}}$ be the associated Clifford algebra with its natural $\mathbb{Z}_2$-gradation. We adopt the convention that the product in $\Cl(V)$ satisfies $vu+uv=-2(v,u)\mathbf 1$ for any $v,u\in V$.

Let $W$ be a $\Cl(V)$-module. The action of an element $c\in\Cl(V)$ on $s\in W$ will be denoted by $c\cdot s$. We will denote by $N\geq 1$ the number of irreducible $\Cl(V)$-components of $W$, counted with their multiplicities (we note that this convention is not universally adopted, and some authors use ``$N$'' to denote the number of irreducible $\so(V)$-components).
  
Fix  a nondegenerate bilinear form $\beta$ on $W$ with the properties:
\begin{enumerate}
\item[(B1)]
$\beta$ is $\so(V)$-invariant,
\item[(B2)]
$\beta$ is symmetric or skew-symmetric 
(we let $\epsilon=-1$ in the former case and $\epsilon=1$ in the latter),
\item[(B3)]
for all $v\in V$ and $s,t\in W$, $\beta(v\cdot s,t)=\epsilon\beta(s,v\cdot t)$.
\end{enumerate}
Set $\m_{-2}=V$, $\m_{-1}=W$, $\m=\m_{-2}+\m_{-1}$. On $\m$ we define a structure of graded Lie algebra with Lie bracket defined, for $s,t\in W$ and $v\in V$, by
\begin{equation}\label{equazione_bracket}
([s,t],v)=\beta(v\cdot s,t).
\end{equation}
\begin{definition}
Any graded Lie algebra 
$\m=\m(V,W,\beta)$ as defined above is called an {\it extended translation algebra}.
\end{definition}

We recall that the \emph{maximal transitive prolongation} of $\m$ is a graded Lie algebra $\g=\sum_{p\in \Z} \g_p$ such that:
\begin{enumerate}
\item
$\g_p$ is finite dimensional for every $p\in\Z$;
\item
$\g_p=\m_p$ for $p=-1,-2$ and $\g_p=0$ for $p<-2$;
\item
for all $p\geq 0$, if $X\in\g_p$ is an element such that $[X,\g_{-1}]=0$, then $X=0$ (\emph{transitivity});
\item
$\g$ is maximal with these properties, i.e. if $\g'$ is another graded Lie algebra satisfying (1), (2), and (3), then there exists an injective homomorphism of graded Lie algebras $\phi\colon\g'\to\g$.
\end{enumerate}
The existence and uniqueness of $\g$ is proved in \cite{N1,N2}.
\begin{remark}
The maximal transitive prolongation depends only on the graded Lie algebra structure of $\m$. We are not aware of any classification of extended translation algebras {\it up to graded Lie algebra isomorphisms}:
Theorems \ref{N1c} and \ref{N1r} easily imply such a classification for $N=1$. 
\end{remark}

We will continue denoting by $W$ (resp. $V$) the spaces $\m_{-1}$ and $\g_{-1}$ (resp. $\m_{-2}$ and $\g_{-2}$). 
For any linear map $\phi\colon W\to W$, we denote by ${}^T\phi\colon W\to W$ the transpose of $\phi$ w.r.t.\ $\beta$.
Similarly, if $\phi\colon V\to V$ is a linear map, we denote by ${}^T\phi\colon V\to V$ the transpose of $\phi$ w.r.t.\ $(\cdot,\cdot)$.

\subsection{Computing $\g_0$}
In this section we explicitly compute the subalgebra $\g_0$ of the maximal transitive prolongation $\g=\sum_{p\in\Z}\g_p$ of an extended translation algebra $\m=V+W$, and we prove that if $\dim V\geq3$ then $\g$ is finite dimensional.

We recall that $\g_0$ is just the Lie algebra of $0$-degree derivations of $\m$. We will use indifferently the notations $[D,X]$ and $DX$ for the Lie bracket in $\g$ of an element $D\in\g_0$ and an element $X\in\m$.

 Let now $E\in\g_0$ be the derivation acting with eigenvalue $-1$ on $W$ and $-2$ on $V$. We call $E$ the \emph{grading element} of $\g$.
Moreover, let  $\h_0$ be the set of elements in $\g_0$ acting trivially on $V$:
\[\h_0=\{D\in\g_0\mid [D,v]=0\ \forall v\in V\}.\]
 
\begin{theorem}
\label{lo0}
The Lie algebra $\g_0$ is a direct sum of ideals:
\[
\g_0=\so(V)\oplus\K E\oplus\h_0\,.
\]
\end{theorem}
\begin{proof}

Note that a $0$-degree linear endomorphism $D\colon\m\to\m$ belongs to $\g_0$ if and only if for all $s\in W$ and $v\in V$ we have:
\begin{equation}\label{eq:trans}
({}^TDv)\cdot s=v\cdot(Ds)+{}^TD(v\cdot s)\ .
\end{equation}
and that $\so(V)$, with the standard action on $W$ and $V$, is contained in $\g_0$. Let $\g_0^A$ (resp. $\g_0^S$) be the set of elements $D\in\g_0$ such that $D|_{V}$ is skew-symmetric (resp. symmetric).
Clearly $\so(V)\subset\g_0^A$ and $\g_0=\so(V)+\g_0^S$ as direct sum of vector spaces. Moreover
$\g_0^A=\so(V)+\h_0$ as direct sum of vector spaces.

We now show that if $D\in\g_0^A$, then ${}^TD\in\g_0^A$.
More precisely, if $D\in\so(V)$ then ${}^TD=-D\in\so(V)$ because $(\,\cdot\,,\,\cdot\,)$ and $\beta$ are $\so(V)$-invariant. 
On the other hand, if $D\in\h_0$ then by \eqref{eq:trans}, for all $v\in V$ and $s\in W$:
\[
v\cdot Ds+{}^TD(v\cdot s)=0\,.
\]
Multiplying by $v$ on both sides,
we also have, for all $v\in V$ and $s\in W$:
\[
v\cdot {}^TDs+D( v\cdot s)=0\,,
\]
and, again by \eqref{eq:trans}, also ${}^TD\in\h_0$.

Let $A\in\so(V)$ be an element of the form $(vu-uv)\in\Lambda^2 V\subset \Cl(V)$, and $H\in\h_0$. 
For any $s\in W$ we have:
\begin{align*}
AHs&=v\cdot u\cdot Hs-u\cdot v\cdot Hs=-v\cdot{}^TH (u\cdot s)+u\cdot{}^TH (v\cdot s)\\
&=H(v\cdot u\cdot s)-H(u\cdot v\cdot s)=HAs.
\end{align*}
Thus $[A,H]=0$ and $[\so(V),\h_0]=0$.
The Lie algebra $\g_0^A$ decomposes hence as direct sum $\g_0^A=\so(V)\oplus \h_0$ of ideals.

Let us now consider $D\in\g_0^S$. By \eqref{eq:trans} we have, for all $v\in V$ and $s\in W$:
\[ Dv\cdot s=v\cdot Ds+{}^TD( v\cdot s)\ .\]
Then, for any $v\in V$ such that $(v,v)\neq0$ and $s\in W$, we get:
\[\frac{v\cdot Dv\cdot v}{(v,v)}\cdot s= -v\cdot {}^TDs-D( v\cdot s)\ .\]
The right hand side is manifestly linear in $v$. 
Expanding the left hand side, in the direction of a non-isotropic vector $u\in V$, orthogonal to $v$, and annihilating the second order term we obtain:
\[ (v,v)\big(v\cdot Du\cdot u+u\cdot Dv\cdot u +u\cdot Du\cdot v\big)=(u,u)(v\cdot Dv\cdot v),\]
and hence
\[0=(u,Du)(v,v)-(v,Dv)(u,u)-2(v,Du)uv\,.\]
It follows that $D$ is a multiple of the identity on $V$, and $D\in\h_0\oplus\mathbb KE$. Since $E$ is in the center of $\g_0$, the theorem follows.
\end{proof}

We observe, for later use in section \ref{sec2}, that $\h_0$ splits into the sum of its symmetric part $\h_0^{s}$ and skew-symmetric part $\h_0^{a}$ with respect to $\beta$. The following relations hold:
\[
[\h_0^{a},\h_0^{a}]\subseteq  \h_0^{a}\ ,\quad[\h_0^{a},\h_0^{s}]\subseteq  \h_0^{s}\ ,\quad[\h_0^{s},\h_0^{s}]\subseteq  \h_0^{a}\ ,
\]
and the condition that elements of $\h_0$ act as derivations of $\m$ yields:
\begin{equation}\label{eeeeaaaa}
\begin{aligned}
\h_0^{a}&=\{D\in\mathfrak{gl}(W)\mid \beta(Ds,t)=-\beta(s,Dt),\     D(v\cdot s)=v\cdot Ds\ \forall v\in V,\ s,t\in W\}\,,\\
\h_0^{s}&=\{D\in\mathfrak{gl}(W)\mid \beta(Ds,t)=\beta(s,Dt),\ D(v\cdot s)=-v\cdot Ds\ \forall v\in V,\ s,t\in W\}\,.
\end{aligned}
\end{equation}
\medskip

We conclude the paragraph proving that, if $\dim V\geq 3$, then the maximal transitive prolongation $\g$ is finite dimensional and the action of $\g_p$, $p>0$ is faithful on $\g_{-2}$. 

\begin{theorem}
\label{fine}
If $\dim V\geq 3$ then
\begin{enumerate}
\item
$\m$ is of finite type, i.e. the maximal transitive prolongation $\g$ of $\m$ has finite dimension;
\item
for all $p>0$ and $X\in\g_p$, if $[X,\g_{-2}]=0$ then $X=0$.
\end{enumerate}
\end{theorem}
\begin{proof}
Let $x,y,z\in V$ be orthogonal non-isotropic vectors. Consider the bilinear form $\alpha$ on $W$ defined by\footnote{This bilinear form has been already considered in \cite{C}.}
\begin{equation}
\label{laforma}
\alpha(s,t)=([y\cdot z\cdot s,t],x)=\beta(x\cdot y \cdot z \cdot s,t)\ .
\end{equation}
Straightforward computations show that 
\begin{enumerate}
\item
$\alpha$ is symmetric and nondegenerate,
\item
$\h_0\subset\so(W,\alpha)$.
\end{enumerate}
Then $\h_0\subset\mathfrak{gl}(W)$ has a trivial Cartan prolongation (in the sense of \textit{e.g.} \cite[Chapter VII]{St}). From the classical Tanaka theory \cite[Theorem 11.1]{N2} the statements follow.
\end{proof}

In the low dimensional cases $\dim V=1,2$, the Lie algebra $\g$ is actually infinite dimensional. This will be discussed in detail in section \ref{sec3}.
\medskip

{\it From now on, we will assume without further mention that $\dim V\geq 3$.}

\subsection{The structure of $\g$ in the $N=1$ case}
In this section we define a canonical $\so(V)$-equivariant linear embedding of $\g_1$ into $W$, such that the image is a $\Cl(V)$-submodule. We then use it to describe the maximal transitive prolongation $\g$ in the case where $N=1$, i.e. $W$ is an irreducible $\Cl(V)$-module.

Let $D\in\g_1$. By Theorem \ref{fine}, the element $D$ is uniquely determined by its action on $V$. 
\begin{lemma}
\label{lemmag1}
For every $D\in\g_1$ there exists a unique $s\in W$ such that $Dv=v\cdot s$ for all $v\in V$.
\end{lemma}
\begin{proof}
Fix an element $u\in V$ with $(u,u)\neq0$ and define a map $D':V\to W$ by
\[
D'v=-\frac{v\cdot u\cdot Du}{(u,u)}\ .
\]
If $v=u$ then clearly $Dv=D'v$. If $v$ is orthogonal to $u$ then, for all $t\in W$:
\begin{align*}
\beta(D'v,t)&=\frac{\beta(u\cdot v\cdot Du,t)}{(u,u)}=\epsilon\frac{\beta(v\cdot Du,u\cdot t)}{(u,u)}\\
  &=\epsilon\frac{([Du,u\cdot t],v)}{(u,u)}=\epsilon\frac{([D(u\cdot t),u],v)}{(u,u)}\ ,
  \end{align*}
where, in the last equality, we used the fact that $[V,W]=0$. Hence,
 \begin{align*}
\beta(D'v,t) &=\epsilon\frac{([{}^T(D(u\cdot t)),v],u)}{(u,u)}=-\epsilon\frac{([D(u\cdot t),v],u)}{(u,u)}\\
  &=\epsilon \frac{([u\cdot t,Dv],u)}{(u,u)}=\beta(Dv,t)\ ,
\end{align*}
where, in the second equality, we used the fact that $D(u\cdot t)$ is the sum of a skew-symmetric operator on $V$ and a multiple of the identity. 
It follows that $D|_{V}=D'$ and the statement is true with $s=-\frac{u\cdot Du}{(u,u)}$.
\end{proof}
Let $\phi\colon\g_1\to W$ be the map satisfying
\[
Dv=v\cdot \phi(D)
\]
for all $D\in\g_1$ and $v\in V$.
It is unique and well defined by Lemma \ref{lemmag1}, and it is injective by (2) in Theorem \ref{fine}. 

For any $A\in \so(V)$, $D\in\g_1$ and $v\in V$, we have
\begin{align*}
v\cdot \phi([A,D])&=[[A,D],v]=[A,[D,v]]-[D,[A,v]]\\
&=A(v\cdot \phi(D))-(Av)\cdot \phi(D)=v\cdot [A,\phi(D)]
\end{align*}
so that $\phi$ is a morphism of $\so(V)$-modules.
\par
The next step is determining the behaviour of $\phi$ with respect to Clifford multiplications. 
\begin{proposition}
\label{hanome}
For every non-isotropic $v\in V$, there exists a $0$-degree Lie algebra automorphism $\psi:\g\to\g$ satisfying
\begin{enumerate}
\item
$\psi(s)=v\cdot s$ for all $s\in W$,
\item
$
\psi(\phi(D))=\epsilon\phi(\psi^{-1}(D))
$
for all $D\in \g_1$.
\end{enumerate}
In particular, the image $\phi(\g_1)\subset W$ is a $\Cl(V)$-submodule of $W$.
\end{proposition}
\begin{proof}
Fix a non-isotropic vector $v\in V$. Define a $0$-degree Lie algebra automorphism $\psi$ of $\mathfrak{m}$ as:
\begin{align*}
\psi(s)&=v\cdot s&&\forall s\in W\ ,\\
\psi(u)&=\epsilon(v,v)\left(u-\frac{2(v,u)}{(v,v)}v\right)&&\forall u\in V\ ,
\end{align*}
and still denote by $\psi$ its canonical extension to a $0$-degree Lie algebra automorphism of $\g$. Observe that:
\begin{align*}
\psi^{-1}(s)&=-\frac{v\cdot s}{(v,v)}=-\frac{\psi(s)}{(v,v)}&&\forall s\in W\ ,\\
\psi^{-1}(u)&=\frac\epsilon{(v,v)}\left(u-\frac{2(v,u)}{(v,v)}v\right)=\frac{\psi(u)}{(v,v)^2}&&\forall u\in V\ .
\end{align*}

If $D\in\g_{1}$, then:
\begin{align*}
u\cdot\phi(\psi^{-1}(D))&=[\psi^{-1}(D),u]=\psi^{-1}([D,\psi(u)])\\
  &=\psi^{-1}\big(\psi(u)\cdot\phi(D)\big)=\epsilon\psi^{-1}\big(v\cdot u\cdot v\cdot\phi(D)\big)\\
  &=\epsilon u\cdot v\cdot\phi(D)
\end{align*}
for all $u\in V$. It follows then
\[
\phi(\psi^{-1}(D))=\epsilon v\cdot\phi(D)=\epsilon\psi(\phi(D))\ .
\]
The last statement follows from the fact that $V$ generates $\Cl(V)$.
\end{proof}
The following theorem describes the structure of $\g$ for $N=1$.
\begin{theorem}
\label{nonloso}
Let $\dim V\geq 3$ and $\g$ be the maximal transitive prolongation of $\m=V+W$. If $W$ is an irreducible $\Cl(V)$-module, then exactly one of the following two cases occurs:
\begin{enumerate}
\item $\g_{p}=0$ for all $p\geq 1$;
\item $\g$ is simple.
\end{enumerate}
\end{theorem}
\begin{proof}
If $\g_{1}=0$, then $\g_{p}=0$ for all $p\geq 1$ by transitivity of $\g$.

Suppose $\g_1\neq 0$.
We first prove that the radical $\ra$ of $\g$ is zero. The proof is a variant on the proof of Theorem $3.22$ in \cite{MN}. Since $\g$ contains a grading element, every ideal of $\g$ is graded. We distinguish two cases.

Case $\ra_{-1}=0$. By transitivity of $\g$, we have $\ra=\ra_{-2}$.
On the other hand, since $\ra_{-2}$ is $\so(V)$-invariant and $\g_{-2}$ is $\so(V)$-irreducible, either $\ra=0$ or $\ra=\g_{-2}$. In the latter case,
$[\g_1,\g_{-2}]\subset \ra_{-1}=0$ and it follows that $\g_1=0$, contradicting our assumptions.

Case $\ra_{-1}\neq 0$. Define by recurrence $\ra^{(0)}=\ra$ and, for $h\geq 1$, $\ra^{(h)}:=[\ra^{(h-1)},\ra^{(h-1)}]$. Let $\ell$ be the smallest integer such that $\ra^{(\ell)}_{-1}\neq 0$ and $\ra^{(\ell+1)}_{-1}= 0$. Note that $\ra^{(\ell)}$ is invariant under the action of all automorphisms of $\g$. In particular, $\ra^{(\ell)}_{-1}$ is invariant under the action of the automorphism
$\psi:\g \to\g$ defined in Proposition \ref{hanome}. It follows that $\ra^{(\ell)}_{-1}$ is a $\Cl(V)$-submodule of $W$ and thus $\ra^{(\ell)}_{-1}=\g_{-1}$. Then, 
\[
\g_{-2}=[\g_{-1},\g_{-1}]=[\ra^{(\ell)}_{-1},\ra^{(\ell)}_{-1}]\subset \ra^{(\ell+1)}_{-2}\subset \g_{-2}\ .
\]
Hence $\ra^{(\ell+1)}_{-2}=\g_{-2}$. Reasoning as in previous case, $[\g_{1},\g_{-2}]\subset \ra^{(\ell+1)}_{-1}=0$, yielding a contradiction.

Therefore $\g$ is semi-simple.  It is simple, because if it is the direct sum of two semi-simple ideals $\g=\g'\oplus\g''$, then each of the subspaces $\g'_{-2}$ and $\g''_{-2}$ is invariant for the action of $\so(V)\subset \g_0$. By irreducibility of $V=\g_{-2}$, one obtains, say, $\g'_{-2}=\g_{-2}$ and $\g''_{-2}=0$. 

By non-degeneracy of the Killing forms of $\g'$ and $\g''$, we also get $\g'_{2}=\g_{2}$ and $\g''_{2}=0$. Hence $\g''\subset \g_{-1}+\g_{0}+\g_{1}$ and, by non-degeneracy of $\beta$, $\g''_{-1}=0=\g''_{1}$. Finally, $\g''=\g''_{0}$ and equality $[\g''_0,\g_{-1}]\subset \g''_{-1}=0$ gives, by transitivity of $\g$, that $\g''=0$. The theorem is proved. 
\end{proof}

\begin{remark}
In the case $N>1$ we do not know if there are examples such that $\g_1\neq0$ and $\phi$ is not surjective, and in particular $\g$ is not simple.
\end{remark}

\section{Classification of simple prolongations.}
\label{sec2}
In this section we classify all real and complex simple graded Lie algebras $\mathfrak{g}$ that arise as prolongations of extended translation algebras $\m=V + W$, for $\dim V\geq 3$. Note that Lemma 3.17 and Theorem 3.21 of \cite{MN}, together with our Theorem \ref{fine}, directly imply that any prolongation $\mathfrak{g}$ of this kind is also necessarily the unique \emph{maximal} and \emph{transitive} prolongation of $\mathfrak{m}$.

\subsection{The complex case}
\label{subsub1}
In some proofs, we will make extensive use of the notation and results of \cite{AC} and the classification of Clifford algebras as given in \cite{LM}. In particular, we adopt the following conventions.

\begin{itemize}
\item 
The symbol $\mathbb S$ will denote an irreducible complex $\Cl(V)$-module. If $\dim V$ is odd, there exist two inequivalent irreducible $\Cl(V)$-modules. Whenever necessary, we denote them by $\mathbb S'$ and $\mathbb S''$. In all other cases, one can safely assume for instance that $\mathbb S$ is the $\Cl(V)$-module explicitly described in \cite{AC}. Clifford multiplication on $\mathbb S$ will be denoted by 
``$\,\circ\,$''. 
\item
If $\mathbb{S}$ is $\so(V)$-reducible, we denote by $\mathbb S^+$ and $\mathbb S^-$ its irreducible $\so(V)$-submodules.
\item
A nondegenerate bilinear form $b\colon\mathbb S\times\mathbb S\to\C$ is called \emph{admissible} if there exist $\tau,\sigma\in\{\pm1\}$ such that $b(v\circ s, t)=\tau b(s, v\circ t)=\sigma b(t, v\circ s)$ for all $v\in V$ and $s,t\in\mathbb{S}$. If $\mathbb S=\mathbb S^++\mathbb S^-$, we also require that $\mathbb S^+$ and $\mathbb S^-$ are either isotropic or mutually orthogonal w.r.t. $b$
(in the former case we set $\imath=-1$, in the latter $\imath=1$). The numbers $(\tau,\sigma)$, or $(\tau,\sigma,\imath)$, are the \emph{invariants associated with $b$}.
\end{itemize}

We first recall the description of simple complex graded Lie algebras (see e.g. \cite{CS, GOV, Y}). Let $\g$ be a complex simple Lie algebra. Fix a Cartan subalgebra $\mathfrak{c}\subset\g$, and denote by $\Delta=\Delta(\g,\mathfrak{c})$ the root system and by $\mathfrak{c}_{\R}\subset\mathfrak{c}$ the real subspace where all the roots are real valued. 
Any element $\lambda\in(\mathfrak{c}_\R^*)^*\simeq\mathfrak{c}_\R$ with $\lambda(\alpha)\in\mathbb{Z}$ for all $\alpha\in\Delta$ defines a gradation on $\g$ by setting:
\begin{align*}
\begin{cases}
\g_0=\mathfrak{c}+
\displaystyle\sum_{\substack{\alpha\in\Delta\\\lambda(\alpha)=0}}\g^\alpha,\\
\g_p=\displaystyle\sum_{\substack{\alpha\in\Delta\\\lambda(\alpha)=p}}\g^\alpha, &\forall p\in\Z\setminus\{0\},
\end{cases}
\end{align*}
and all possible gradations on $\g$ are of this form, for some choice of $\mathfrak{c}$ and $\lambda$. We will refer to $\lambda(\alpha)$ as the \emph{degree} of the root $\alpha$.

There exists a set of positive roots $\Delta^+\subset\Delta$ such that $\lambda$ is dominant, i.e. $\lambda(\alpha)\geq0$ for all $\alpha\in\Delta^+$. 
The \emph{depth} of $\g$ is the degree of the maximal root.

Let $\Pi$ be the set of positive simple roots, which we identify with the nodes of the Dynkin diagram. A gradation is \emph{fundamental} if $\g_{-1}$ generates $\sum_{p<0}\g_p$.
This is true if and only if $\lambda(\alpha)\in\{0,1\}$ for all simple roots $\alpha\in\Pi$. We denote a fundamental gradation on a Lie algebra $\g$ by marking with a cross the nodes of the Dynkin diagram of $\g$ corresponding to simple roots $\alpha$ with $\lambda(\alpha)=1$.

The Lie subalgebra $\g_0$ is reductive. The Dynkin diagram of its semisimple ideal is obtained from the Dynkin diagram of $\g$ by removing all crossed nodes, and any line issuing from them.

A routine examination of the Dynkin diagrams of simple complex Lie algebras shows that Table \ref{tabella} below lists all the fundamental complex graded simple Lie algebras satisfying the conditions:
\begin{enumerate}
\item
the gradation has depth $2$;
\item
$\g_0$ is the Lie algebra direct sum of $\so(n,\C)$, where $n=\dim\g_{-2}\geq3$, of $\C E$, where $E$ is the grading element, and of a reductive Lie subalgebra $\h_0$ acting trivially on $\g_{-2}$;
\item
the representation of $\so(n,\C)\subset\g_0$ on $\g_{-2}$ is equivalent to the standard representation of $\so(n,\C)$.
\end{enumerate}
\vfill

\begin{table}[H]
\begin{tabular}{|c|c|c|c|c|}
\hline
\hline
& $\g$ & $\dim\g_{-2}$ & $\dim\g_{-1}$ & $\h_0$\\
\hline
\hline

$A_{\ell}$&
\begin{tikzpicture}
\node[root]   (1)                     {};
\node[xroot] (2) [right=of 1] {} edge [-] (1);
\node[]   (3) [right=of 2] {$\;\cdots\,$} edge [-] (2);
\node[xroot]   (4) [right=of 3] {} edge [-] (3);
\node[root]   (5) [right=of 4] {} edge [-] (4);
\end{tikzpicture}&
$4$ & $4(\ell-3)$ &$A_{\ell-4}\oplus\C$
\\
\hline

$B_4$&
\begin{tikzpicture}
\node[root]   (1)                     {};
\node[root] (2) [right=of 1] {} edge [-] (1);
\node[root]   (3) [right=of 2] {} edge [-] (2);
\node[xroot]   (4) [right=of 3] {} edge [rdoublearrow] (3);
\end{tikzpicture}&
$6$ & $4$ & $0$
\\
\hline

$B_{\ell}$&
\begin{tikzpicture}
\node[root]   (1)                     {};
\node[root] (2) [right=of 1] {} edge [-] (1);
\node[root]   (3) [right=of 2] {} edge [-] (2);
\node[xroot]   (4) [right=of 3] {} edge [-] (3);
\node[]   (5) [right=of 4] {$\;\cdots\,$} edge [-] (4);
\node[root]   (7) [right=of 5] {} edge [rdoublearrow] (5);
\end{tikzpicture}&
$6$ & $8\ell-28$ & $B_{\ell-4}$
\\
\hline 

$C_{\ell}$&
\begin{tikzpicture}
\node[root]   (1)                     {};
\node[xroot] (2) [right=of 1] {} edge [-] (1);
\node[]   (3) [right=of 2] {$\;\cdots\,$} edge [-] (2);
\node[root]   (5) [right=of 3] {} edge [doublearrow] (3);
\end{tikzpicture}&
$3$ & $4(\ell-2)$ & $C_{\ell-2}$
\\
\hline

$D_5$&
\begin{tikzpicture}
\node[root]   (1)                     {};
\node[root] (2) [right=of 1] {} edge [-] (1);
\node[root]   (3) [right=of 2] {} edge [-] (2);
\node[xroot]   (4) [above right=of 3] {} edge [-] (3);
\node[xroot]   (5) [below right=of 3] {} edge [-] (3);
\end{tikzpicture}&
$6$ & $8$ & $\C$
\\
\hline

$D_{\ell}$&
\begin{tikzpicture}
\node[root]   (1)                     {};
\node[root] (2) [right=of 1] {} edge [-] (1);
\node[root]   (3) [right=of 2] {} edge [-] (2);
\node[xroot]   (4) [right=of 3] {} edge [-] (3);
\node[]   (6) [right=of 4] {$\;\cdots\quad$} edge [-] (4);
\node[root]   (7) [above right=of 6] {} edge [-] (6);
\node[root]   (8) [below right=of 6] {} edge [-] (6);
\end{tikzpicture}&
$6$ & $8(\ell-4)$ & $D_{\ell-4}$
\\
\hline

$E_6$&
\begin{tikzpicture}
\node[xroot]   (1)                     {};
\node[root] (2) [right=of 1] {} edge [-] (1);
\node[root]   (3) [right=of 2] {} edge [-] (2);
\node[root]   (4) [right=of 3] {} edge [-] (3);
\node[xroot]   (5) [right=of 4] {} edge [-] (4);
\node[root]   (6) [below=of 3] {} edge [-] (3);
\end{tikzpicture}&
$8$ & $16$ & $\C$
\\
\hline

$E_7$&
\begin{tikzpicture}
\node[root]   (1)              {};
\node[root]   (3) [right=of 1] {} edge [-] (1);
\node[root]   (4) [right=of 3] {} edge [-] (3);
\node[root]   (5) [right=of 4] {} edge [-] (4);
\node[xroot]  (6) [right=of 5] {} edge [-] (5);
\node[root]   (7) [right=of 6] {} edge [-] (6);
\node[root]   (2) [below=of 4] {} edge [-] (4);
\end{tikzpicture}&
$10$ & $32$ & $A_1$
\\
\hline

$E_8$&
\begin{tikzpicture}
\node[xroot]  (1)              {};
\node[root]   (3) [right=of 1] {} edge [-] (1);
\node[root]   (4) [right=of 3] {} edge [-] (3);
\node[root]   (5) [right=of 4] {} edge [-] (4);
\node[root]   (6) [right=of 5] {} edge [-] (5);
\node[root]   (7) [right=of 6] {} edge [-] (6);
\node[root]   (2) [below=of 4] {} edge [-] (4);
\node[root]   (8) [right=of 7] {} edge [-] (7);
\end{tikzpicture}&
$14$ & $64$ & $0$
\\
\hline

$F_4$&
\begin{tikzpicture}
\node[root]   (1)                     {};
\node[root] (2) [right=of 1] {} edge [-] (1);
\node[root]   (3) [right=of 2] {} edge [rdoublearrow] (2);
\node[xroot]   (4) [right=of 3] {} edge [-] (3);
\end{tikzpicture}&
$7$ & $8$ & $0$\\
\hline 
\hline
\end{tabular}
\caption{}
\label{tabella}
\end{table}

%
%
%
%

\begin{theorem}
\label{theoC}
The simple complex graded Lie algebras $\g$ that are the maximal transitive prolongations of  a complex extended translation algebra are the following:
\begin{table}[H]
\begin{tabular}{|c|c|c|c|c|c|}
\hline
\hline
& $\g$ & $\dim\g_{-2}$ & $\dim\g_{-1}$ & $N$ & $\h_0$\\
\hline
\hline

$\begin{gathered}\mathrm{A}_\ell\\\ell\geq4\end{gathered}$&\begin{tikzpicture}
\node[root]   (1)                     {};
\node[xroot] (2) [right=of 1] {} edge [-] (1);
\node[]   (3) [right=of 2] {$\;\cdots\,$} edge [-] (2);
\node[xroot]   (4) [right=of 3] {} edge [-] (3);
\node[root]   (5) [right=of 4] {} edge [-] (4);
\end{tikzpicture}&
$4$ & $4N$ & $\ell-3$ &$\mathfrak{sl}(N,\C)\oplus\C$
\\
\hline

$\begin{gathered}\mathrm{C}_\ell\\\ell\geq3\end{gathered}$&
\begin{tikzpicture}
\node[root]   (1)                     {};
\node[xroot] (2) [right=of 1] {} edge [-] (1);
\node[]   (3) [right=of 2] {$\;\cdots\,$} edge [-] (2);
\node[root]   (5) [right=of 3] {} edge [doublearrow] (3);
\end{tikzpicture}&
$3$ & $2N$ & $2(\ell-2)$ & $\mathfrak{sp}(\frac{N}{2},\C)$
\\
\hline

$\mathrm E_6$&
\begin{tikzpicture}
\node[xroot]   (1)                     {};
\node[root] (2) [right=of 1] {} edge [-] (1);
\node[root]   (3) [right=of 2] {} edge [-] (2);
\node[root]   (4) [right=of 3] {} edge [-] (3);
\node[xroot]   (5) [right=of 4] {} edge [-] (4);
\node[root]   (6) [below=of 3] {} edge [-] (3);
\end{tikzpicture}&
$8$ & $16$ & $1$ & $\C$
\\
\hline

$\mathrm F_4$&
\begin{tikzpicture}
\node[root]   (1)                     {};
\node[root] (2) [right=of 1] {} edge [-] (1);
\node[root]   (3) [right=of 2] {} edge [rdoublearrow] (2);
\node[xroot]   (4) [right=of 3] {} edge [-] (3);
\end{tikzpicture}&
$7$ & $8$ & $1$ & $0$
\\
\hline 
\hline
\end{tabular}
\caption{}
\label{tabellina}
\end{table}
\end{theorem}
In this table, the multiplicity of the spinor representation $\mathbb{S}$ in $\g_{-1}$ is denoted by $N$. 
Moreover, we adopt the convention that the Lie algebra $\mathrm{C}_m=\mathfrak{sp}(m,\mathbb{C})$ acts naturally on $\C^{2m}$.

\begin{proof}
Conditions (1), (2) and (3) above are clearly necessary for $\g$ to be the prolongation of an extended translation algebra.

Among the Lie algebras listed in table \ref{tabella}, we select those such that 
the representation of $\so(n,\C)\subset\g_0$ on $\g_{-1}$ is equivalent to a multiple of the spinor representation $\mathbb S$. In this way we obtain exactly the Lie algebras listed in table \ref{tabellina}.

To conclude the proof, we show in each case that the negatively graded part $\g_{-2}+\g_{-1}$ is isomorphic, as graded Lie algebra, to an extended translation algebra $\m(V,W,\beta)$.
In fact, for suitable identifications $\g_{-1}\simeq\mathbb{S}\oplus\cdots\oplus\mathbb{S}$ and $\g_{-2}=V$, we exhibit a nondegenerate bilinear form $\beta$ on $\g_{-1}$ satisfying (B1), (B2), (B3) and such that the Lie bracket on $\g_{-1}$ is of the form \eqref{equazione_bracket}. 

\paragraph
{($A_\ell$)} As $(\so(4,\mathbb{C})\oplus\mathfrak{sl}(N,\C))$-module, $\g_{-1}$ is isomorphic to $\mathbb{S^+}\otimes \mathbb{C}^{N}+\mathbb{S^-}\otimes (\mathbb{C}^{N})^*$.
The Lie bracket is induced by a $(\so(4,\mathbb{C})\oplus\mathfrak{sl}(N,\C))$-equivariant map
$
(\mathbb{S^+}\otimes \mathbb{C}^{N}) \times (\mathbb{S^-}\otimes (\mathbb{C}^{N})^*)\rightarrow V
$.
Hence, there exists a $\so(4,\mathbb{C})$-equivariant map $\Gamma:\mathbb{S^+}\times \mathbb{S^-}\rightarrow V$ such that
\[
[s^+\otimes c,s^-\otimes c^*]=\Gamma(s^+,s^-)c^*(c)\ ,
\]
for any $s^{\pm}\in\mathbb{S^\pm}$, $c\in\mathbb{C}^N$ and $c^*\in(\mathbb{C}^N)^*$.
By the results of \cite{AC}, the map $\Gamma$ is uniquely determined by a nondegenerate bilinear form $b:\mathbb{S}\times\mathbb{S}\rightarrow\mathbb{C}$ with invariants $(\tau,\sigma,\imath)=(+,-,+)$ by $(\Gamma(s^+,s^-),v)=b(v\circ s^+,s^-)$.
\\
Identifying $\mathbb{C}^N$ and $(\mathbb{C}^N)^*$ via a nondegenerate symmetric bilinear form $\delta$ on $\mathbb{C}^N$,
we define a Clifford multiplication ``$\,\cdot\,$''$\colon\g_{-2}\otimes \g_{-1}\rightarrow \g_{-1}$ by:
\begin{equation}
\label{CliffA}
v\cdot (s\otimes c)=(v\circ s)\otimes c\ ,
\end{equation}
and a nondegenerate $\so(4,\C)$-invariant bilinear form $\beta:\g_{-1}\times\g_{-1}\rightarrow\mathbb{C}$:
\[
\beta(s\otimes c,t\otimes d)=b(s,t)\delta(c,d)\ .
\]
Direct computations show that $\cdot$ and $\beta$ satisfy all the required conditions.
\paragraph
{($C_\ell$)} In this case, $\g_{-1}\simeq\mathbb{S}\otimes \mathbb{C}^{N}$ as a $(\so(3,\mathbb{C})\oplus\mathfrak{sp}(\frac{N}{2},\C))$-module and
the Lie bracket is given by a $(\so(3,\mathbb{C})\oplus\mathfrak{sp}(\frac{N}{2},\C))$-equivariant map
$
\vee^2\mathbb{S}\otimes\wedge^2\mathbb{C}^{N}\rightarrow V
$.
Hence, there exists a $\so(3,\mathbb{C})$-equivariant map $\Gamma:\mathbb{S}\vee \mathbb{S}\rightarrow V$ and a nondegenerate $\mathfrak{sp}(\frac{N}{2},\mathbb{C})$-invariant bilinear form $\omega$ on $\mathbb{C}^N$ such that
\[
[s\otimes c,t\otimes d]=\Gamma(s,t)\omega(c,d)\ ,
\]
for any $s,t\in\mathbb{S}$ and $c,d\in\mathbb{C}^N$. The map $\Gamma$ is uniquely determined by a nondegenerate bilinear form $b:\mathbb{S}\times\mathbb{S}\rightarrow\mathbb{C}$ with invariants $(\tau,\sigma)=(-,-)$ by $(\Gamma(s,t),v)=b(v\circ s,t)$.
\\
We again define a Clifford multiplication by \eqref{CliffA} and a nondegenerate $(\so(3,\mathbb{C})\oplus\mathfrak{sp}(\frac{N}{2},\C))$-invariant bilinear form $\beta=b\otimes\omega:\g_{-1}\times\g_{-1}\rightarrow\mathbb{C}$.
\paragraph
{($E_6$)} 
In this case, $\g_{-1}\simeq\mathbb{S}$ as a $\so(8,\mathbb{C})$-module and
the Lie bracket is uniquely determined by a nondegenerate bilinear form $b:\mathbb{S}\times\mathbb{S}\rightarrow\mathbb{C}$ with invariants $(\tau,\sigma)=(-,+)$ by $([s,t],v)=b(v\circ s,t)$.
\paragraph
{($F_4$)} 
In this case, $\g_{-1}\simeq\mathbb{S}$ as a $\so(7,\mathbb{C})$-module and
the Lie bracket is uniquely determined by a nondegenerate bilinear form $b:\mathbb{S}\times\mathbb{S}\rightarrow\mathbb{C}$ with invariants $(\tau,\sigma)=(-,+)$ by $([s,t],v)=b(v\circ s,t)$.
\end{proof}
We show that in the case $(A_\ell)$ with $N=2$, a different choice of $\beta$ yields the same extended translation algebra.
This result will be needed in the following subsection.
\begin{lemma}
\label{uaoh}
Let $\dim V=4$ and $\mathbb{S}=\mathbb{S}^++\mathbb{S}^-$ be the complex irreducible $\Cl(V)$-module. 
Let $f_E$ and $f$ be the unique (up to constant) bilinear forms on $\mathbb{S}$ with invariants 
$(\tau,\sigma,\imath)(f_E)=(+,-,+)$ and $(\tau,\sigma,\imath)(f)=(-,-,+)$.
\\
Denote by $\phi_j\colon\mathbb{S}\to\mathbb{S}_j$, $j=1,2$, two isomorphisms of $\Cl(V)$-modules.
The extended translation algebras $\mathfrak{m}$ and $\mathfrak{n}$ on $V+(\mathbb{S}_1+\mathbb{S}_2)$, associated with the bilinear forms $\beta_{\mathfrak{m}}$ and $\beta_{\mathfrak{n}}$, given on $\mathbb{S}_1+\mathbb{S}_2$  by
\begin{align*}
&\beta_{\mathfrak{m}}=(f_E\circ\phi_1)+ (f_E\circ\phi_2)\ ,\\
\bigg\{
&\begin{aligned}
  &\beta_{\mathfrak{n}}(\mathbb{S}_1,\mathbb{S}_1)=\beta_{\mathfrak{n}}(\mathbb{S}_2,\mathbb{S}_2)=0\,,\\
  &\beta_{\mathfrak{n}}(\phi_1(s),\phi_2(t))=\beta_{\mathfrak{n}}(\phi_2(t),\phi_1(s))=f(s,t)\,,
\end{aligned}
\end{align*}
are graded-isomorphic via a linear automorphism of $V+(\mathbb{S}_1+\mathbb{S}_2)$ acting trivially on $V$.
\end{lemma}
\begin{proof}
Define a graded isomorphism $\varphi:\mathfrak{m}\rightarrow\mathfrak{n}$ by $\varphi|_{V}=id$ and
\[
\varphi|_{\mathbb{S}^-_1}=id\,,\quad\varphi|_{\mathbb{S}^-_2}=i\,id\,,\quad\varphi|_{\mathbb{S}^+_1}=\phi_2\circ\phi_1^{-1}\,\quad\varphi|_{\mathbb{S}^+_2}=i\,\phi_1\circ\phi_2^{-1}\ .
\]
To check the statement, recall that $f_E(\cdot,\cdot)=f(\mathcal{E}(\cdot),\cdot)$ with $\mathcal{E}|_{\mathbb{S^\pm}}=\pm id$.
\end{proof}
\par\noindent
\subsection{The real case}
\label{subsub2}
We use the same notation as in \cite{AC}:

\begin{itemize}
\item 
The symbol $S$ will denote an irreducible real $\Cl(V)$-module. When there exist two inequivalent irreducible $\Cl(V)$-modules and we need to distinguish them, we denote them by $S'$ and $S''$. In all other cases, one can safely assume for instance that $S$ is the $\Cl(V)$-module explicitly described in \cite{AC}. Clifford multiplication on $S$ will be denoted by 
``$\,\circ\,$''. 
\item
If ${S}$ is $\so(V)$-reducible, we denote by $S^+$ and $S^-$ its irreducible $\so(V)$-submodules.
\item
Admissible bilinear forms and their invariants are defined as in the complex case.
\end{itemize}

We recall the description of simple real graded Lie algebras (see e.g. \cite{CS, D, GOV}). Let $\g$ be a real simple Lie algebra. Fix a Cartan decomposition $\g=\mathfrak{k}+\mathfrak{p}$, a maximal abelian subspace $\a\subset\mathfrak p$, and a maximal torus $\t$ in the centralizer of $\a$ in $\mathfrak k$. Then $\c=\a+\t$ is a Cartan subalgebra of $\g$.

Denote by $\Delta=\Delta(\g^\C,\c^\C)$ the root system of $\g^\C$ and by $\c_{\R}=\a+\mathrm{i}\t\subset\c^\C$ the real subspace where all the roots have real values. 
Conjugation of $\g^\C$ with respect to the real form $\g$ leaves $\c^\C$ invariant, and induces an 
involution $\alpha\mapsto\bar\alpha$ on $\c_\R^*$, trasforming roots into roots.
We say that a root $\alpha$ is compact if $\bar\alpha=-\alpha$ and denote by $\Delta_\bullet$ the set of compact roots.
There exists a set of positive roots $\Delta^+\subset\Delta$, with corresponding system of simple roots $\Pi$, and an involutive automorphism $\varepsilon\colon\Pi\to\Pi$ of the Dynkin diagram of $\g^\C$, such that $\epsilon(\Pi\setminus\Delta_\bullet)\subseteq \Pi\setminus\Delta_\bullet$ and
\[
\bar\alpha=-\alpha \text{ for all }\alpha\in\Pi\cap\Delta_\bullet \ ,
\]
\[
\bar\alpha=\varepsilon(\alpha)+\sum_{\beta\in\Pi\cap\Delta_\bullet}b_{\alpha,\beta}\beta\text{ for all }\alpha\in\Pi\setminus\Delta_\bullet\ .
\]
The \emph{Satake diagram} of $\g$ is the Dynkin diagram of $\g^\C$ with the following additional information:
\begin{enumerate}
\item
nodes in $\Pi\cap\Delta_\bullet$ are painted black;
\item
if $\alpha\in\Pi\setminus\Delta_\bullet$ and $\varepsilon(\alpha)\neq\alpha$ then $\alpha$ and $\varepsilon(\alpha)$ are joined by a curved arrow.
\end{enumerate}
A list of Satake diagrams can be found in \cite{CS, GOV, O}.

Let $\lambda\in(\c_\R^*)^*\simeq\c_\R$ be 
an element such that the induced gradation on $\g^\C$ is fundamental. Then the gradation on $\g^\C$ induces a gradation on $\g$ if and only if $\bar\lambda=\lambda$, or equivalently the two following conditions  on the set $\Phi=\{\alpha\in\Pi\mid\lambda(\alpha)=1\}$ are satisfied:
\begin{enumerate}
\item
$\Phi\cap\Delta_\bullet=\emptyset$;
\item
if $\alpha\in\Phi$ then $\varepsilon(\alpha)\in\Phi$.
\end{enumerate}
We indicate the gradation on a real Lie algebra by crossing all nodes in $\Phi$.

In the real case too the Lie subalgebra $\g_0$ is reductive and the Satake diagram of its semisimple ideal is obtained from the Satake diagram of $\g$ by removing all crossed nodes, and any line issuing from them.
\medskip

The first step in the classification is giving necessary conditions on the complexification of a simple maximal transitive prolongation of a real extended translation algebra.
\begin{proposition}
Let $\m=\m(V,W,\beta)$ be a real extended translation algebra, with simple maximal transitive prolongation $\g$. Then the complexification $\m^\C=\m(V^\C,W^\C,\beta^\C)$ of $\m$ is a complex extended translation algebra with simple complex maximal transitive prolongation equal to the complexification $\g^\C$ of $\g$.
\end{proposition}
\begin{proof}
The assertion that $\m^\C$ is a complex extended translation algebra is straightforward. The proposition then easily follows from Lemmata \ref{blablabla} and \ref{blublublu} below.
\end{proof}

\begin{lemma}\label{blablabla}
If $\m=\m_{-2}+\m_{-1}$ is a real fundamental graded Lie algebra and $\g$ is its maximal transitive prolongation, then the complexification  $\g^\C$ of $\g$ is the complex maximal transitive prolongation of the complexification $\m^\C$ of $\m$.
\end{lemma}
\begin{proof}
Let $\tilde\g$ be the complex maximal transitive prolongation of $\m^\C$ . Since $\g^\C$ is a complex transitive prolongation of $\m^\C$, it is clear that $\g^\C\subset\tilde\g$.

Let $\sigma\colon\m^\C\to\m^\C$ be the antilinear involution given by complex conjugation with respect to the real form $\m$. Then $\sigma$ uniquely extends to an antilinear involution $\tilde\sigma\colon\tilde\g\to\tilde\g$. The set $\tilde\g^{\tilde\sigma}$ of fixed points of $\tilde\sigma$ in $\tilde\g$ is a transitive prolongation of $\m$, hence it is contained in $\g$.
Since $\tilde\g=(\tilde\g^{\tilde\sigma})^\C$, we obtain $\tilde\g\subset\g^\C$ , and finally $\tilde\g=\g^\C$.
\end{proof}

\begin{lemma}\label{blublublu}
If $\m=\m(V,W,\beta)$ is a real extended translation algebra and  its maximal transitive prolongation $\g$ is simple then the complexification $\g^\C$ of $\g$ is also simple.
\end{lemma}
\begin{proof}
We remind that we assumed $\dim V\geq 3$, hence $V$ is $\so(V)$-irreducible and $V^\C$ is $\so(V^\C)$-irreducible. 
This in particular implies that $\g_{-2}$ (resp. $\g^\C_{-2}$) is $\g_0$-irreducible (resp. $\g_0^\C$-irreducible). 
If $\g^\C$ is not simple then, arguing as in the proof of Theorem \ref{nonloso}, one gets that $\g_{-2}$ is not $\g_0$-irreducible, and hence a contradiction.
\end{proof}
\enlargethispage{3mm}

Table \ref{tabellona} lists the real graded simple Lie algebras $\g$ such that $\g^\C$ is the maximal transitive prolongation of a complex extended translation algebra.

\noindent

\begin{table}[H]
\begin{tabular}{|c|c|c|c|c|c|}
\hline
\hline
& $\g$ & $\begin{gathered}\mathrm{sign}\,\g_{-2}\\(+,-)\end{gathered}$ & $\dim\g_{-1}$ & $\h_0$\\
\hline
\hline

$\begin{gathered}\mathrm{A_\ell I}\\\ell\geq4\end{gathered}$&
\begin{tikzpicture}
\node[root]   (1)                     {};
\node[xroot] (2) [right=of 1] {} edge [-] (1);
\node[]   (3) [right=of 2] {$\;\cdots\,$} edge [-] (2);
\node[xroot]   (4) [right=of 3] {} edge [-] (3);
\node[root]   (5) [right=of 4] {} edge [-] (4);
\end{tikzpicture}&
$(2,2)$ & $4(\ell-3)$ &$\mathfrak{sl}(\ell-3,\R)\oplus\R$
\\
\hline

$\begin{gathered}\mathrm{A_\ell II}\\\ell\geq5\\\ell\text{ odd}\end{gathered}$&
\begin{tikzpicture}
\node[broot]   (1)                     {};
\node[xroot] (2) [right=of 1] {} edge [-] (1);
\node[]   (3) [right=of 2] {$\;\cdots\,$} edge [-] (2);
\node[xroot]   (4) [right=of 3] {} edge [-] (3);
\node[broot]   (5) [right=of 4] {} edge [-] (4);
\end{tikzpicture}&
$\begin{gathered}(4,0)\\(0,4)\end{gathered}$ & $4(\ell-3)$ &$\mathfrak{sl}(\ell-3/2,\mathbb H)\oplus\R$
\\
\hline

$\begin{gathered}\mathrm{A_\ell III}\\\ell\geq4\end{gathered}$&
\begin{tikzpicture}
\node[root]   (1)                     {}; 
\node[xroot] (2) [right=of 1] {} edge [-] (1);
\node[]   (3) [right=of 2] {$\;\cdots\,$} edge [-] (2);
\node[xroot]   (4) [right=of 3] {} edge [-] (3) edge [<->,out=150, in=30] (2);
\node[root]   (5) [right=of 4] {} edge [-] (4) edge [<->,out=145, in=35] (1);
\end{tikzpicture}&
$(3,1)$ & $4(\ell-3)$ & $\begin{gathered}\mathfrak{su}(p,q)\oplus\R\\ p\leq q,\, p+q=\ell-3\end{gathered}$
\\
\hline

$\begin{gathered}\mathrm{A_\ell III}\\\ell\geq4\end{gathered}$&
\begin{tikzpicture}
\node[root]   (1)                     {}; 
\node[xroot] (2) [right=of 1] {} edge [-] (1);
\node[]   (3) [right=of 2] {$\;\cdots\,$} edge [-] (2);
\node[xroot]   (4) [right=of 3] {} edge [-] (3) edge [<->,out=150, in=30] (2);
\node[root]   (5) [right=of 4] {} edge [-] (4) edge [<->,out=145, in=35] (1);
\end{tikzpicture}&
$(1,3)$ & $4(\ell-3)$ & $\begin{gathered}\mathfrak{su}(p,q)\oplus\R\\ p\leq q,\, p+q=\ell-3\end{gathered}$
\\
\hline

$\begin{gathered}\mathrm{C_\ell I}\\\ell\geq3\end{gathered}$&
\begin{tikzpicture}
\node[root]   (1)                     {};
\node[xroot] (2) [right=of 1] {} edge [-] (1);
\node[]   (3) [right=of 2] {$\;\cdots\,$} edge [-] (2);
\node[root]   (5) [right=of 3] {} edge [doublearrow] (3);
\end{tikzpicture}&
$(2,1)$ & $4(\ell-2)$ & $\mathfrak{sp}(\ell-2,\R)$
\\
\hline

$\begin{gathered}\mathrm{C_\ell I}\\\ell\geq3\end{gathered}$&
\begin{tikzpicture}
\node[root]   (1)                     {};
\node[xroot] (2) [right=of 1] {} edge [-] (1);
\node[]   (3) [right=of 2] {$\;\cdots\,$} edge [-] (2);
\node[root]   (5) [right=of 3] {} edge [doublearrow] (3);
\end{tikzpicture}&
$(1,2)$ & $4(\ell-2)$ & $\mathfrak{sp}(\ell-2,\R)$
\\
\hline

$\begin{gathered}\mathrm{C_\ell II}\\\ell\geq3\end{gathered}$&
$\begin{gathered}
\begin{tikzpicture}
\node[broot]   (1)                     {};
\node[xroot] (2) [right=of 1] {} edge [-] (1);
\node[]   (3) [right=of 2] {$\;\cdots\,$} edge [-] (2);
\node[broot]   (5) [right=of 3] {} edge [doublearrow] (3);
\end{tikzpicture}\\
\begin{tikzpicture}
\node[broot]   (1)                     {};
\node[xroot] (2) [right=of 1] {} edge [-] (1);
\node[]   (3) [right=of 2] {$\;\cdots\,$} edge [-] (2);
\node[root]   (5) [right=of 3] {} edge [doublearrow] (3);
\end{tikzpicture}\end{gathered}$&
$(3,0)$ & $4(\ell-2)$ & $\begin{gathered}\mathfrak{sp}(p,q)\\p\leq q, \,p+q=\ell-2\end{gathered}$
\\
\hline

$\begin{gathered}\mathrm{C_\ell II}\\\ell\geq3\end{gathered}$&
$\begin{gathered}
\begin{tikzpicture}
\node[broot]   (1)                     {};
\node[xroot] (2) [right=of 1] {} edge [-] (1);
\node[]   (3) [right=of 2] {$\;\cdots\,$} edge [-] (2);
\node[broot]   (5) [right=of 3] {} edge [doublearrow] (3);
\end{tikzpicture}\\
\begin{tikzpicture}
\node[broot]   (1)                     {};
\node[xroot] (2) [right=of 1] {} edge [-] (1);
\node[]   (3) [right=of 2] {$\;\cdots\,$} edge [-] (2);
\node[root]   (5) [right=of 3] {} edge [doublearrow] (3);
\end{tikzpicture}\end{gathered}$&
$(0,3)$ & $4(\ell-2)$ & $\begin{gathered}\mathfrak{sp}(p,q)\\p\leq q, \,p+q=\ell-2\end{gathered}$
\\
\hline

$\mathrm{E\,I}$&
\begin{tikzpicture}
\node[xroot]   (1)                     {};
\node[root] (2) [right=of 1] {} edge [-] (1);
\node[root]   (3) [right=of 2] {} edge [-] (2);
\node[root]   (4) [right=of 3] {} edge [-] (3);
\node[xroot]   (5) [right=of 4] {} edge [-] (4);
\node[root]   (6) [below=of 3] {} edge [-] (3);
\end{tikzpicture}&
$(4,4)$ & $16$ & $\R$
\\
\hline

$\mathrm{E\,II}$&
\begin{tikzpicture}
\node[xroot]   (1)                     {};
\node[root] (2) [right=of 1] {} edge [-] (1);
\node[root]   (3) [right=of 2] {} edge [-] (2);
\node[root]   (4) [right=of 3] {} edge [-] (3)  edge [<->,out=150, in=30] (2);
\node[xroot]   (5) [right=of 4] {} edge [-] (4)  edge [<->,out=145, in=35] (1);
\node[root]   (6) [below=of 3] {} edge [-] (3);
\end{tikzpicture}&
$(5,3)$ & $16$ & $\R$
\\
\hline

$\mathrm{E\,II}$&
\begin{tikzpicture}
\node[xroot]   (1)                     {};
\node[root] (2) [right=of 1] {} edge [-] (1);
\node[root]   (3) [right=of 2] {} edge [-] (2);
\node[root]   (4) [right=of 3] {} edge [-] (3)  edge [<->,out=150, in=30] (2);
\node[xroot]   (5) [right=of 4] {} edge [-] (4)  edge [<->,out=145, in=35] (1);
\node[root]   (6) [below=of 3] {} edge [-] (3);
\end{tikzpicture}&
$(3,5)$ & $16$ & $\R$
\\
\hline

$\mathrm{E\,III}$&
\begin{tikzpicture}
\node[xroot]   (1)                     {};
\node[broot] (2) [right=of 1] {} edge [-] (1);
\node[broot]   (3) [right=of 2] {} edge [-] (2);
\node[broot]   (4) [right=of 3] {} edge [-] (3);
\node[xroot]   (5) [right=of 4] {} edge [-] (4)  edge [<->,out=150, in=30] (1);
\node[root]   (6) [below=of 3] {} edge [-] (3);
\end{tikzpicture}&
$(7,1)$ & $16$ & $\R$
\\
\hline

$\mathrm{E\,III}$&
\begin{tikzpicture}
\node[xroot]   (1)                     {};
\node[broot] (2) [right=of 1] {} edge [-] (1);
\node[broot]   (3) [right=of 2] {} edge [-] (2);
\node[broot]   (4) [right=of 3] {} edge [-] (3);
\node[xroot]   (5) [right=of 4] {} edge [-] (4)  edge [<->,out=150, in=30] (1);
\node[root]   (6) [below=of 3] {} edge [-] (3);
\end{tikzpicture}&
$(1,7)$ & $16$ & $\R$
\\
\hline

$\mathrm{E\,IV}$&
\begin{tikzpicture}
\node[xroot]   (1)                     {};
\node[broot] (2) [right=of 1] {} edge [-] (1);
\node[broot]   (3) [right=of 2] {} edge [-] (2);
\node[broot]   (4) [right=of 3] {} edge [-] (3);
\node[xroot]   (5) [right=of 4] {} edge [-] (4);
\node[broot]   (6) [below=of 3] {} edge [-] (3);
\end{tikzpicture}&
$\begin{gathered}(8,0)\\(0,8)\end{gathered}$ & $16$ & $\R$
\\
\hline

$\mathrm{F\,I}$&
\begin{tikzpicture}
\node[root]   (1)                     {};
\node[root] (2) [right=of 1] {} edge [-] (1);
\node[root]   (3) [right=of 2] {} edge [rdoublearrow] (2);
\node[xroot]   (4) [right=of 3] {} edge [-] (3);
\end{tikzpicture}&
$(4,3)$ & $8$ & $0$
\\
\hline 

$\mathrm{F\,I}$&
\begin{tikzpicture}
\node[root]   (1)                     {};
\node[root] (2) [right=of 1] {} edge [-] (1);
\node[root]   (3) [right=of 2] {} edge [rdoublearrow] (2);
\node[xroot]   (4) [right=of 3] {} edge [-] (3);
\end{tikzpicture}&
$(3,4)$ & $8$ & $0$
\\
\hline 

$\mathrm{F\,II}$&
\begin{tikzpicture}
\node[broot]   (1)                     {};
\node[broot] (2) [right=of 1] {} edge [-] (1);
\node[broot]   (3) [right=of 2] {} edge [rdoublearrow] (2);
\node[xroot]   (4) [right=of 3] {} edge [-] (3);
\end{tikzpicture}&
$(7,0)$ & $8$ & $0$
\\
\hline 

$\mathrm{F\,II}$&
\begin{tikzpicture}
\node[broot]   (1)                     {};
\node[broot] (2) [right=of 1] {} edge [-] (1);
\node[broot]   (3) [right=of 2] {} edge [rdoublearrow] (2);
\node[xroot]   (4) [right=of 3] {} edge [-] (3);
\end{tikzpicture}&
$(0,7)$ & $8$ & $0$
\\
\hline 

\hline
\end{tabular}
\caption{}
\label{tabellona}
\end{table}

We prove now the main classification result of this section.
\begin{theorem}
\label{abcdefg}
The simple real graded Lie algebras $\g$ that are the maximal transitive prolongations of a real extended translation algebra are the following:
\begin{table}[H]
\begin{tabular}{|c|c|c|c|c|c|c|}
\hline
\hline
& $\g$ & $\begin{gathered}\mathrm{sign}\,\g_{-2}\\(+,-)\end{gathered}$ & $\dim\g_{-1}$ & $N$ & $\h_0$\\
\hline
\hline

$\begin{gathered}\mathrm{A_\ell I}\\\ell\geq4\end{gathered}$&
\begin{tikzpicture}
\node[root]   (1)                     {};
\node[xroot] (2) [right=of 1] {} edge [-] (1);
\node[]   (3) [right=of 2] {$\;\cdots\,$} edge [-] (2);
\node[xroot]   (4) [right=of 3] {} edge [-] (3);
\node[root]   (5) [right=of 4] {} edge [-] (4);
\end{tikzpicture}&
$(2,2)$ & $4N$ & $\ell-3$ &$\mathfrak{sl}(N,\R)\oplus\R$
\\
\hline

$\begin{gathered}\mathrm{A_\ell II}\\\ell\geq5\\\ell\text{ odd}\end{gathered}$&
\begin{tikzpicture}
\node[broot]   (1)                     {};
\node[xroot] (2) [right=of 1] {} edge [-] (1);
\node[]   (3) [right=of 2] {$\;\cdots\,$} edge [-] (2);
\node[xroot]   (4) [right=of 3] {} edge [-] (3);
\node[broot]   (5) [right=of 4] {} edge [-] (4);
\end{tikzpicture}&
$\begin{gathered}(4,0)\\(0,4)\end{gathered}$ & $8N$ & $\frac{\ell-3}{2}$ &$\mathfrak{sl}(N,\mathbb H)\oplus\R$
\\
\hline

$\begin{gathered}\mathrm{A_\ell III}\\\ell\geq4\\\ell\text{ odd}\end{gathered}$&
\begin{tikzpicture}
\node[root]   (1)                     {}; 
\node[xroot] (2) [right=of 1] {} edge [-] (1);
\node[]   (3) [right=of 2] {$\;\cdots\,$} edge [-] (2);
\node[xroot]   (4) [right=of 3] {} edge [-] (3) edge [<->,out=150, in=30] (2);
\node[root]   (5) [right=of 4] {} edge [-] (4) edge [<->,out=145, in=35] (1);
\end{tikzpicture}&
$(3,1)$ & $8N$ & $\frac{\ell-3}{2}$ & $\begin{gathered}\mathfrak{su}(p,q)\oplus\R\\ p\leq q,\, p+q=2N\end{gathered}$
\\
\hline

$\begin{gathered}\mathrm{A_\ell III}\\\ell\geq4\end{gathered}$&
\begin{tikzpicture}
\node[root]   (1)                     {}; 
\node[xroot] (2) [right=of 1] {} edge [-] (1);
\node[]   (3) [right=of 2] {$\;\cdots\,$} edge [-] (2);
\node[xroot]   (4) [right=of 3] {} edge [-] (3) edge [<->,out=150, in=30] (2);
\node[root]   (5) [right=of 4] {} edge [-] (4) edge [<->,out=145, in=35] (1);
\end{tikzpicture}&
$(1,3)$ & $4N$ & $\ell-3$ & $\begin{gathered}\mathfrak{su}(p,q)\oplus\R\\ p\leq q,\, p+q=N\end{gathered}$
\\
\hline

$\begin{gathered}\mathrm{C_\ell I}\\\ell\geq3\end{gathered}$&
\begin{tikzpicture}
\node[root]   (1)                     {};
\node[xroot] (2) [right=of 1] {} edge [-] (1);
\node[]   (3) [right=of 2] {$\;\cdots\,$} edge [-] (2);
\node[root]   (5) [right=of 3] {} edge [doublearrow] (3);
\end{tikzpicture}&
$(2,1)$ & $4N$ & $\ell-2$ & $\mathfrak{sp}(N,\R)$
\\
\hline

$\begin{gathered}\mathrm{C_\ell I}\\\ell\geq3\end{gathered}$&
\begin{tikzpicture}
\node[root]   (1)                     {};
\node[xroot] (2) [right=of 1] {} edge [-] (1);
\node[]   (3) [right=of 2] {$\;\cdots\,$} edge [-] (2);
\node[root]   (5) [right=of 3] {} edge [doublearrow] (3);
\end{tikzpicture}&
$(1,2)$ & $2N$ & $2(\ell-2)$ & $\mathfrak{sp}(\frac{N}{2},\R)$
\\
\hline

$\begin{gathered}\mathrm{C_\ell II}\\\ell\geq3\end{gathered}$&
$\begin{gathered}
\begin{tikzpicture}
\node[broot]   (1)                     {};
\node[xroot] (2) [right=of 1] {} edge [-] (1);
\node[]   (3) [right=of 2] {$\;\cdots\,$} edge [-] (2);
\node[broot]   (5) [right=of 3] {} edge [doublearrow] (3);
\end{tikzpicture}\\
\begin{tikzpicture}
\node[broot]   (1)                     {};
\node[xroot] (2) [right=of 1] {} edge [-] (1);
\node[]   (3) [right=of 2] {$\;\cdots\,$} edge [-] (2);
\node[root]   (5) [right=of 3] {} edge [doublearrow] (3);
\end{tikzpicture}\end{gathered}$&
$(3,0)$ & $4N$ & $\ell-2$ & $\begin{gathered}\mathfrak{sp}(p,q)\\p\leq q, \,p+q=N\end{gathered}$
\\
\hline

$\begin{gathered}\mathrm{C_\ell II}\\\ell\geq3\end{gathered}$&
$\begin{gathered}
\begin{tikzpicture}
\node[broot]   (1)                     {};
\node[xroot] (2) [right=of 1] {} edge [-] (1);
\node[]   (3) [right=of 2] {$\;\cdots\,$} edge [-] (2);
\node[broot]   (5) [right=of 3] {} edge [doublearrow] (3);
\end{tikzpicture}\\
\begin{tikzpicture}
\node[broot]   (1)                     {};
\node[xroot] (2) [right=of 1] {} edge [-] (1);
\node[]   (3) [right=of 2] {$\;\cdots\,$} edge [-] (2);
\node[root]   (5) [right=of 3] {} edge [doublearrow] (3);
\end{tikzpicture}\end{gathered}$&
$(0,3)$ & $4N$ & $\ell-2$ & $\begin{gathered}\mathfrak{sp}(p,q)\\p\leq q, \,p+q=N\end{gathered}$
\\
\hline

$\mathrm{E\,I}$&
\begin{tikzpicture}
\node[xroot]   (1)                     {};
\node[root] (2) [right=of 1] {} edge [-] (1);
\node[root]   (3) [right=of 2] {} edge [-] (2);
\node[root]   (4) [right=of 3] {} edge [-] (3);
\node[xroot]   (5) [right=of 4] {} edge [-] (4);
\node[root]   (6) [below=of 3] {} edge [-] (3);
\end{tikzpicture}&
$(4,4)$ & $16$ & $1$ & $\R$
\\
\hline

$\mathrm{E\,II}$&
\begin{tikzpicture}
\node[xroot]   (1)                     {};
\node[root] (2) [right=of 1] {} edge [-] (1);
\node[root]   (3) [right=of 2] {} edge [-] (2);
\node[root]   (4) [right=of 3] {} edge [-] (3)  edge [<->,out=150, in=30] (2);
\node[xroot]   (5) [right=of 4] {} edge [-] (4)  edge [<->,out=145, in=35] (1);
\node[root]   (6) [below=of 3] {} edge [-] (3);
\end{tikzpicture}&
$(3,5)$ & $16$ & $1$ & $\R$
\\
\hline

$\mathrm{E\,III}$&
\begin{tikzpicture}
\node[xroot]   (1)                     {};
\node[broot] (2) [right=of 1] {} edge [-] (1);
\node[broot]   (3) [right=of 2] {} edge [-] (2);
\node[broot]   (4) [right=of 3] {} edge [-] (3);
\node[xroot]   (5) [right=of 4] {} edge [-] (4)  edge [<->,out=150, in=30] (1);
\node[root]   (6) [below=of 3] {} edge [-] (3);
\end{tikzpicture}&
$(7,1)$ & $16$ & $1$ & $\R$
\\
\hline

$\mathrm{E\,IV}$&
\begin{tikzpicture}
\node[xroot]   (1)                     {};
\node[broot] (2) [right=of 1] {} edge [-] (1);
\node[broot]   (3) [right=of 2] {} edge [-] (2);
\node[broot]   (4) [right=of 3] {} edge [-] (3);
\node[xroot]   (5) [right=of 4] {} edge [-] (4);
\node[broot]   (6) [below=of 3] {} edge [-] (3);
\end{tikzpicture}&
$\begin{gathered}(8,0)\\(0,8)\end{gathered}$ & $16$ & $1$ & $\R$
\\
\hline

$\mathrm{F\,I}$&
\begin{tikzpicture}
\node[root]   (1)                     {};
\node[root] (2) [right=of 1] {} edge [-] (1);
\node[root]   (3) [right=of 2] {} edge [rdoublearrow] (2);
\node[xroot]   (4) [right=of 3] {} edge [-] (3);
\end{tikzpicture}&
$(3,4)$ & $8$ & $1$ & $0$
\\
\hline 

$\mathrm{F\,II}$&
\begin{tikzpicture}
\node[broot]   (1)                     {};
\node[broot] (2) [right=of 1] {} edge [-] (1);
\node[broot]   (3) [right=of 2] {} edge [rdoublearrow] (2);
\node[xroot]   (4) [right=of 3] {} edge [-] (3);
\end{tikzpicture}&
$(7,0)$ & $8$ & $1$ & $0$
\\
\hline 
\hline
\end{tabular}
\caption{}
\label{tabella_reale}
\end{table}
\end{theorem}
As in the complex case, the multiplicity of $S$ in $\g_{-1}$ is denoted by $N$.

\begin{proof}
We first look at the exceptional Lie algebras. In cases $\mathrm{E\,II}$ with signature $(5,3)$, $\mathrm{E\,III}\,(1,7)$, $\mathrm{F\,I}\,(4,3)$, and $\mathrm{F\,II}\,(0,7)$ the dimension of the irreducible $\Cl(V)$-module $S$ equals twice the dimension of $\g_{-1}$, and hence their negatively graded parts are not extended translation algebras.
In all the remaining cases (i.e. $\mathrm{E\,I}\,(4,4)$, $\mathrm{E\,II}\,(3,5)$, $\mathrm{E\,III}\,(7,1)$, $\mathrm {E\,IV}\,(8,0)$, $\mathrm {E\,IV}\,(0,8)$, $\mathrm{F\,I}\,(3,4)$, and $\mathrm{F\,II}\,(7,0)$), the dimension of an irreducible $\Cl(V)$-module $S$ equals the dimension of $\g_{-1}$ and its complexification $\mathbb S=S^\C$ is a complex irreducible $\Cl(V^\C)$-module. Moreover, we know from \cite{AC} that, for each of these signatures on $V$, there is exactly one equivalence class of real extended translation algebras with $N=1$, and exactly one equivalence class of complex extended translation algebras with $N=1$. From the classification in the complex case and from Lemma \ref{blablabla} we obtain the result.
\smallskip

\noindent
We now focus on the classical cases. For expository reasons, we first consider the case $N=1$ (or $N=2$ if $\g=\mathrm{C_{3}\,I}\,(1,2)$).


\subsubsection*{Case $\mathrm{A_{4}\,I}$.} The signature of  $\g_{-2}$ is $(2,2)$ and $\dim \g_{-1}=(\ell-3)\dim S$. We consider the case $\ell=4$, i.e. $N=1$. There is exactly one equivalence class of real extended translation algebras, determined by a bilinear form
$b=f_{E}:S\otimes S\rightarrow\R$ with invariants $(+,-,+)$. 

There is only one extended complex translation algebra in dimension $4$, with prolongation $\mathrm{A}_{4}$. Hence, by Lemma \ref{blablabla}, $\g^\C=\mathrm{A}_{4}$ and  $\g=\mathrm{A_4 I}$.

\subsubsection*{Case $\mathrm{A_{5}\,II}$.} The signature of  $\g_{-2}$ is $(4,0)$ or $(0,4)$. We consider the case $\ell=5$, i.e. $N=1$. 
The $i$-linear extension to $S^\C=S+ i S$ of the complex structure $I\in \mathrm{End}(S)$ (see \cite{AC}) determines a $\Cl(V^\C)$-invariant decomposition
$
S^\C=S^{1,0}\oplus S^{0,1}
$
with $S^{1,0}\simeq S^{0,1}\simeq\mathbb{S}$ irreducible $\Cl(V^\C)$ modules.

Consider the forms
$b_{(4,0)}=h_{J}:S\otimes S\rightarrow\R$ if $\mathrm{sign}\,\g_{-2}=(4,0)$, and $b_{(0,4)}=h_{EJ}:S\otimes S\rightarrow\R$ if $\mathrm{sign}\,\g_{-2}=(0,4)$. Both have invariants $(+,-)$ and
$
S^{1,0}\perp S^{0,1}
$
with respect to the $i$-linear extensions of $b_{(4,0)}$ and $b_{(0,4)}$.

By comparing with the description of the complex case $\mathrm{A}_{5}$, we see that $\g^{\C}=\mathrm{A}_{5}$ so that $\g=\mathrm{A_5 II}$. Note that the complexified extended translation algebra has $N=2$.

\subsubsection*{Case $\mathrm{A_{4}\,III}\,(1,3)$.} 
For $\ell=4$, i.e. $N=1$, there is exactly one equivalence class of real extended translation algebras, determined by the bilinear form
$b=f_{E}\otimes h_{J}$ with invariants $(+,-)$. 
Again one can check that $\g^{\C}=\mathrm{A}_{4}$ so that $\g=\mathrm{A_4 III}$.

\subsubsection*{Case $\mathrm{A_{5}\,III}\,(3,1)$.}
We look at the case $N=1$, i.e. $\ell= 5$, and $\g\simeq\mathfrak{su}(3,3)$ or $\g\simeq\mathfrak{su}(2,4)$.
In both cases, 
the standard $\Cl(V)$-module $S$ is isomorphic (see again \cite{AC}) to the tensor product $S_{1,1}\otimes S_{2,0}$ of the standard $\Cl(\R^{1,1})$- and $\Cl(\R^{2,0})$-modules.
The space of $\so(V)$-invariant bilinear forms on $S$ with invariants $\tau$ and $\sigma$ satisfying $\tau\sigma=-1$ is $4$-dimensional. 
The $i$-linear extension to $S^\C=S+ i S$ of the complex structure $\Id\otimes I\in End(S)$ determines a decomposition
$
S^\C=S^{1,0}\oplus S^{0,1}
$
which is stable under $\Cl(V^\C)$, i.e. $S^{1,0}\simeq S^{0,1}\simeq\mathbb{S}$ as $\Cl(V^\C)$-irreducible modules.

Consider the bilinear forms 
$b^1=f_{E}\otimes h_{J}$ and $b^2=f\otimes h_{EI}$ on $S$ with invariants $(+,-)$ and $(-,+)$, respectively. Then, with respect to the $i$-linear extensions of $b^1$ (resp. $b^2$),
we have $S^{1,0}\perp_{b^1} S^{0,1}$ (resp. $S^{1,0}$ and $S^{0,1}$ are $b^2$-isotropic).

By comparing with the description of the complex case $\mathrm{A}_{5}$, using also Lemma \ref{uaoh}, we see that $\g^{\C}=\mathrm{A}_{5}$ in both cases. We observe that the complexified extended translation algebra has $N=2$.

Explicit computations give 
\[ \h_0^{b^1}\simeq\mathfrak{su}(1,1)\oplus\R\,,\qquad\h_0^{b^2}\simeq \mathfrak{su}(2)\oplus\R\,, \]
so that $\g\simeq\mathfrak{su}(3,3)$ or $\g\simeq\mathfrak{su}(2,4)$, respectively.
\smallskip


\subsubsection*{Case $\mathrm{C_{3}\,I}\,(1,2)$.} We look at the case $N=2$, because $N=1$ does not correspond to a simple prolongation of extended translation algebra. 
The standard $\Cl(V)$-module $S$ is isomorphic to the tensor product $S_{1,1}\otimes S_{0,1}$ and
there is exactly one admissible bilinear form on $S$ (up to scalar), namely
$b=f\otimes h$, with invariants $(-,-)$. 
Let $\omega$ be the standard symplectic form on $\R^{2}$ and consider the extended translation algebra
$ \mathfrak{m}:=\R^{1,2}+S\otimes \R^{2} $
with Lie bracket determined by the Clifford multiplication
\[
v\cdot(s\otimes c)=(v\circ s)\otimes c
\]
for any $v\in \R^{1,2}$, $c\in\R^2$, and by the bilinear form 
$\beta=b\otimes\omega$, with invariants $(-,+)$.

One can check that $\g^{\C}=\mathrm{C}_{3}$ so that $\g=\mathrm{C_3 I}$.

\subsubsection*{Case $\mathrm{C_{3}\,I}\,(2,1)$.} In this case $\dim_{\R}S=4$, $S\simeq S_{1,1}\otimes S_{1,0}$ and $N=1$. There is exactly one admissible  $\so(V)$-invariant bilinear form on $S$ with invariants $\tau$ and $\sigma$ satisfying $\tau\sigma=-1$, namely $b=f_E\otimes h_{I}$, with invariants $(+,-,+)$. 

Denote by $\vol\in\Cl(V)$ the volume form. Then  $P=i\,\vol$ is a para-complex structure and
determines a decomposition 
$S^\C=\mathbb{S}'\oplus\mathbb{S}''$
into the two {inequivalent} irreducible $\Cl(V^\C)$-modules $\mathbb S^{'}$ and $\mathbb S^{''}$.
Both $\mathbb S'$ and $\mathbb S''$ are isotropic w.r.t. the complexification of $b$. 
A straightforward computation shows that the complexified fundamental Lie algebra determined by $b$ concides with the one described in the complex case $\mathrm C_3$.

\subsubsection*{Case $\mathrm{C_{3}\,II}\,(0,3)$.} The space of $\so(V)$-invariant bilinear forms on $S$ with $\tau\sigma=-1$ is generated by $b=h_I$ with invariants $(+,-)$. The paracomplex structure given by $i\,\vol\in\Cl(V^\C)$ determines a decomposition
$
S^\C=\mathbb S^{'}\oplus \mathbb S^{''}
$
into two inequivalent $\Cl(V^\C)$-irreducible modules
$\mathbb S^{'}$ and $\mathbb S^{''}$, isotropic w.r.t. the complexification of $b$.
As in the previous case, one can see that $\g^\C=\mathrm C_3$.

\subsubsection*{Case $\mathrm{C_{3}\,II}\,(3,0)$.} The space of $\so(V)$-invariant bilinear forms on $S$ with $\tau\sigma=-1$ is generated by $b=h$ with invariants $(-,+)$.
The endomorphism $iI\in\mathrm{End}(S^\C)$ is a paracomplex structure that determines a decomposition
$S^\C=\mathbb S\oplus \mathbb S$
into two equivalent $\Cl(V^\C)$-irreducible modules, isotropic w.r.t. the complexification of $b$. Again $\g^\C$ is isomorphic to $\mathrm C_3$.

\bigskip

Now we look at higher $N$. We will explicitly describe only the families of algebras $\mathrm{C_{\ell}\,II}\,(3,0)$ and $\mathrm{A_{\ell}\,III}\,(1,3)$. The cases $\mathrm{C_{\ell}\,II}\,(0,3)$ and $\mathrm{A_{\ell}\,III}\,(3,1)$ 
can be treated analogously, whereas the proof of Theorem \ref{theoC} can be carried over unchanged for each of the remaining cases.
\subsubsection*{Case $\mathrm{A_{\ell}\,III}\,(1,3)$.}

In this case $N=\ell-3$ and $W=S\otimes \mathbb{R}^N$ is a $\Cl(V)$-module via 
$
v\cdot (s\otimes c)=(v\circ s)\otimes c
$, for $v\in V=\mathbb{R}^{1,3}$, $s\in S$, $c\in\mathbb{R}^N$.
Let $\delta$ be a nondegenerate symmetric bilinear form on $\mathbb{R}^N$ of signature $(p,q)$, and $\mathfrak{m}_{p,q}=V+W$ the real extended translation algebra associated to the following bilinear form on $W$ with invariants $(+,-)$:
\[
\beta(s\otimes c,t\otimes d):=b(s,t)\delta(c,d)\ .
\]
Again $\g^{\C}=\mathrm{A}_{\ell}$ and the maximal transitive prolongation $\g_{p,q}$ of $\mathfrak{m}_{p,q}$ is of type $\mathrm{A_\ell III}$. We now want to show that $\mathfrak{h}_0(\mathfrak{g_{p,q}})\simeq \mathfrak{u}(p,q)$. 
Let $\alpha$ be the symmetric bilinear form on $W$ defined by \eqref{laforma}, where $x\cdot y\cdot z\in\Cl(\mathbb{R}^{0,3})\subset \Cl(V)$ equals the volume form of $\mathbb{R}^{0,3}$. Explicit computations show that the signature of $\alpha$ is $(4p,4q)$ and that the volume $\vol\in \Cl(V)$, which is a complex structure on $W\simeq \mathbb{C}^2\otimes \mathbb{R}^N$, is in $\so(W,\alpha)$. Recall that $\mathfrak{h}_0\subset \so(W,\alpha)$. Hence, the Hermitian form
\[
H(\hat{w},\tilde{w}):=\alpha(\hat{w},\tilde{w})+i\alpha(\vol\cdot \hat{w},\tilde{w}):W\times W\rightarrow\mathbb{C}\ ,\qquad(\hat{w},\tilde{w}\in W)
\]
has complex signature $(2p,2q)$ and, since $[\mathfrak{h}_0,\Cl(V)^{\overline{0}}]=0$, one also gets that $\mathfrak{h}_0\subset\mathfrak{u}(W,H)\simeq \mathfrak{u}(2p,2q)$. Finally, for any element $s\in S\simeq \mathbb{C}^2$ of $H$-norm 1, consider the following $H$-orthogonal decomposition of $W$
\[ W=W'+ W''\ ,\quad W':=\mathbb{C}s\otimes \mathbb{R}^N\, \quad W'':=\mathbb{C}(x\cdot y\cdot z\cdot s)\otimes\mathbb{R}^N \]
into Hermitian vector spaces of complex signature $(p,q)$. Direct computations show that
$\mathfrak{h}_0$ preserve the decomposition and any of its elements is fully determined by the action on, say, $W'$. This implies $\mathfrak{h}_0\simeq \mathfrak{u}(p,q)$.

\subsubsection*{Case $\mathrm{C_{\ell}\,II}\,(3,0)$.}
In this case $N=\ell-2$ and $W=S\otimes \mathbb{R}^N$ is a $\Cl(V)$-module via 
$
v\cdot (s\otimes c)=(v\circ s)\otimes c
$, for $v\in V=\mathbb{R}^{3,0}$, $s\in S$, $c\in\mathbb{R}^N$.
Let $\delta$ be a nondegenerate symmetric bilinear form on $\mathbb{R}^N$ of signature $(p,q)$, and $\mathfrak{m}_{p,q}=V+W$ the real extended translation algebra associated to the following bilinear form on $W$ with invariants $(-,+)$:
\[
\beta(s\otimes c,t\otimes d):=b(s,t)\delta(c,d)\ .
\]
Again $\g^{\C}=\mathrm{C}_{\ell}$ and the maximal transitive prolongation $\g_{p,q}$ of $\mathfrak{m}_{p,q}$ is of type $\mathrm{C_\ell II}$. We now want to show that $\mathfrak{h}_0(\mathfrak{g_{p,q}})\simeq \mathfrak{sp}(p,q)$. The symmetric bilinear form $\alpha$ on $W$ defined by \eqref{laforma}, with $x\cdot y\cdot z=\vol\in\Cl(V)$, turns out to be proportional to $\beta$. Hence, it is of real signature $(4p,4q)$. Since $\mathfrak{h}_0\subset \so(W,\alpha)=\so(W,\beta)$, we get that $\mathfrak{h}_0=\mathfrak{h}_0^a$ (see \eqref{eeeeaaaa}). Hence
{\scriptsize \[
H(\hat{w},\tilde{w}):=\alpha(\hat{w},\tilde{w})+i\alpha(x\cdot \hat{w},\tilde{w})+j\alpha(y\cdot \hat{w},\tilde{w})+k\alpha(z\cdot \hat{w},\tilde{w}):W\times W\rightarrow\mathbb{H}\ ,\ (\hat{w},\tilde{w}\in W)
\]}
is a $\mathbb{H}$-hermitian form of quaternionic signature $(p,q)$ and $\mathfrak{h}_0\simeq \mathfrak{sp}(p,q)$. 
\end{proof}
\section{The low dimensional cases $\dim V=1,2$.}
\label{sec3}
We finally consider the extended translation algebras $\m=V+W$, when $\dim V=1$ or $2$: Theorem \ref{fine} does not apply in these cases, and in fact the maximal transitive prolongation $\g$ of $\m$ turns out to be always infinite dimensional.

Let us denote by $\mathfrak c_\ell(\C)$ (resp. $\mathfrak c_{\ell}(\R)$) the infinite-dimensional complex (resp. real) contact algebra in dimension $2\ell+1$ (see e.g. \cite{MT} for a detailed description) and by $\mathfrak{n}_\ell(\C)$ (resp. $\mathfrak{n}_\ell(\R)$) its negatively graded part.
\subsection{The complex case}
\begin{proposition}
Let $\m$ be a complex extended translation algebra with $\dim V=2$. Then $N$ is even, $\m$ is isomorphic to $\mathfrak n_{N/2}(\C)\oplus\mathfrak n_{N/2}(\C)$ and the maximal transitive prolongation is $\mathfrak{g}=\mathfrak c_{N/2}(\C)\oplus \mathfrak c_{N/2}(\C)$. 
\end{proposition}
\begin{proof}
We can identify $\so(2,\C)=\C$, $\m_{-1}=\C^N_++\C^N_-$, and $\m_{-2}=\C_{+}+\C_{-}$, with action of $\so(2,\C)$ given by:
\begin{align*}
\lambda\cdot s_\pm=\pm\lambda s_\pm,\quad\lambda\cdot v_\pm=\pm2\lambda v_\pm\quad (\lambda\in\mathfrak{so}(2,\mathbb{C}),\ s_\pm\in \C^N_\pm,\ v_\pm\in \C_\pm).
\end{align*}
Then $\m=\m_+\oplus\m_-$ as a direct sum of ideals. If $N$ is odd then $\m$ cannot be nondegenerate. The proposition follows.
\end{proof}

\begin{proposition}
Let $\m$ be a complex extended translation algebra with $\dim V=1$. Then $N$ is even and the maximal transitive prolongation is $\mathfrak{g}=\mathfrak c_{N/2}(\C)$.
\end{proposition}
\begin{proof}
We can identify $\mathfrak{m}_{-1}=\mathbb{C}^N$, and
$\mathfrak{m}_{-2}=\mathbb{C}$. If $N$ is odd, then $\mathfrak{m}$ cannot be nondegenerate. Then $N$ is even and $\mathfrak{g}=\mathfrak c_{N/2}(\C)$.
\end{proof}

\subsection{The real case}
We first give a lemma.
\begin{lemma}\label{blobloblo}
Let $\m=\sum_{j<0}\m_{j}$ be a complex fundamental nondegenerate graded Lie algebra and $\g$ its maximal complex transitive prolongation. Let $\m_\R$ (resp. $\g_\R$) be the real graded Lie algebra obtained from $\m$ (resp. $\g$) by restriction of scalars. Then $\g_\R$ is the real maximal transitive prolongation of $\m_\R$.
\end{lemma}
\begin{proof}
The complexification $\m_\R^\C$ of $\m_\R$ (resp. $\g_\R^\C$ of $\g_\R$) is isomorphic, as a complex Lie algebra, to the direct sum of two ideals $\m_\R^\C=\m\oplus\overline\m$ (resp. $\g_\R^\C=\g\oplus\overline\g$). By \cite[Proposition 3.3]{MN}, the maximal complex transitive prolongation of $\m\oplus\overline\m$ is $\g\oplus\overline\g$. Then $\g_\R$ is a real transitive prolongation of $\m_\R$ such that its complexification is the maximal complex transitive prolongation of $\m_\R^\C$. By Lemma \ref{blablabla}, $\g_\R$ is the maximal real transitive prolongation of $\m_\R$.
\end{proof}

We give now the description of $\m$ and $\g$ for $\dim V=1,2$.
\begin{proposition}
Let $\m$ be a real extended translation algebra with $\dim V=2$. Then one of the following three cases occurs:
\begin{itemize}
\item If $\sign\,V=(2,0)$, then $N$ is an arbitrary positive integer, $\m=\mathfrak n_N(\C)_{\R}$,  and $\mathfrak{g}=\mathfrak{c}_{N}({\mathbb{C}})_{\R}$.
\item If $\sign\,V=(0,2)$, then $N$ is even, $\m=\mathfrak n_{N/2}(\C)_{\R}$,  and $\mathfrak{g}=\mathfrak{c}_{N/2}({\mathbb{C}})_{\R}$.
\item If $\sign\,V=(1,1)$, then $N$ is even, $\m=\mathfrak n_{N/2}(\R)\oplus \mathfrak n_{N/2}(\R)$,  and $\mathfrak{g}=\mathfrak c_{N/2}(\R)\oplus \mathfrak c_{N/2}(\R)$.
\end{itemize}
\end{proposition}
\begin{proof}
By the description of the complex case, $\m^\C=\mathfrak n(\C)\oplus\mathfrak n(\C)$ as direct sum of ideals.
Let $\varsigma$ be the conjugation with respect to $\m$. There are two possibilities.

If $\varsigma$ exchanges the two ideals isomorphic to $\mathfrak n(\C)$ then $\m=\mathfrak n(\C)_{\R}$ and the real maximal transitive prolongation is $\mathfrak{g}=\mathfrak{c}({\mathbb{C}})_{\R}$ by Lemma \ref{blobloblo}. In this case $\m_{-2}$ is 
$\g_0$-irreducible, hence also 
$\so(V)$-irreducible and the signature of $V$ is $(2,0)$ or $(0,2)$. In signature $(2,0)$ we have that $N=\dim\m_{-1}/4$ is an arbitrary positive integer and $\mathfrak{g}=\mathfrak{c}_{N}({\mathbb{C}})_{\R}$. In signature $(0,2)$ we have that $N=\dim\m_{-1}/2$ must be even and $\mathfrak{g}=\mathfrak{c}_{{N}/{2}}({\mathbb{C}})_{\R}$.

If each of the two ideals is $\varsigma$-invariant, then $\m=\mathfrak{n}(\R)\oplus\mathfrak{n}(\R)$ as direct sum of ideals. In this case $\m_{-2}$ is $\g_0$- and $\so(V)$-reducible and hence the signature of $V$ is $(1,1)$. One has that $N=\dim\m_{-1}/2$ is even and $\mathfrak{g}=\mathfrak{c}_{{N}/{2}}({\mathbb{R}})\oplus\mathfrak{c}_{{N}/{2}}({\mathbb{R}})$.
\end{proof}

\begin{proposition}
Let $\m$ be a real extended translation algebra with $\dim V=1$. 
Then exacltly one of the following two cases occurs:
\begin{itemize}
\item If $\sign V =(1,0)$, then $N$ is an arbitrary positive integer, $\m=\mathfrak n_N(\R)$,  and $\mathfrak{g}=\mathfrak{c}_{N}(\R)$.
\item If $\sign V =(0,1)$, then $N$ is even, $\m=\mathfrak n_{N/2}(\R)$,  and $\mathfrak{g}=\mathfrak{c}_{N/2}(\R)$.
\end{itemize}
\end{proposition}
\begin{proof}
In signature $(1,0)$, one has that $N=\dim \mathfrak{m}_{-1}/ 2$ is an arbitrary positive integer.
In signature $(0,1)$, one has that $N=\dim \mathfrak{m}_{-1}$ is even. The statement easily follows.
\end{proof}

\section{The $N=1$ classification.}
\label{sec4}
In this section we explicitly describe the maximal transitive prolongations of real and complex extended translation algebras with $N=1$. 
\\
In the complex case, we have the following.
\begin{theorem}
\label{N1c}
Let $\mathfrak{m}=\mathfrak{m}_{-2}+\mathfrak{m}_{-1}=V+W$ be a complex extended translation algebra with $N=1$. 
Then $\dim V\equiv 4,5,6,7,8\; \mathrm{mod}\; 8$. If $\dim V=4,7,8$, then the maximal transitive complex prolongation $\g$ of $\mathfrak{m}$ is simple and
\begin{itemize}
\item[--] $\mathfrak{g}=\mathfrak{sl}(5,\mathbb{C})$ if $\dim V=4$,
\item[--] $\mathfrak{g}=\mathrm F_{4}$ if $\dim V=7$,
\item[--] $\mathfrak{g}=\mathrm E_{6}$ if $\dim V=8$,
\end{itemize}
with gradation as described in Theorem \ref{theoC}. \\
In all other dimensions, $\mathfrak{g}$ is not semi-simple, $\mathfrak{g}_0$ is as in Theorem \ref{lo0}, and $\mathfrak{g}_{i}=0$ for any $i\geq 1$.
\end{theorem}
In the real case, we obtain instead the following theorem.
\begin{theorem}
\label{N1r}
Let $V$ be a pseudo-Euclidean vector space of signature $\mathrm{sign} V=(p,q)$ and $n=p+q\;\mathrm{mod}\;8$, $s=p-q\;\mathrm{mod}\;8$.
Let $\mathfrak{m}=\mathfrak{m}_{-2}+\mathfrak{m}_{-1}=V+W$ be a real extended translation algebra with $N=1$. Then 
$(n,s)\neq (1,-1)$, $(2,-2)$, $(2,0)$, $(3,-1)$. If $\dim V=1,2$, then the maximal transitive real prolongation $\g$ is infinite dimensional and
\begin{itemize}
\item[--] $\mathfrak{g}$ is the real contact algebra in dim. $3$ if $\mathrm{sign}\, V=(1,0)$,
\item[--] $\mathfrak{g}$ is the complex contact algebra in dim. $3$ if $\mathrm{sign}\, V=(2,0)$.
\end{itemize}
If $\dim V\geq 3$, then $\g$ is finite dimensional and it is simple exactly in the following cases
\begin{itemize}
\item[--] $\mathfrak{g}=\mathfrak{sp}(1,2)$ if $\mathrm{sign}\, V=(3,0)$ or $(0,3)$,
\item[--] $\mathfrak{g}=\mathfrak{sp}(3,\mathbb{R})$ if $\mathrm{sign}\, V=(2,1)$,
\item[--] $\mathfrak{g}=\mathfrak{sl}(5,\mathbb{R})$ if $\mathrm{sign}\, V=(2,2)$,
\item[--] $\mathfrak{g}=\mathfrak{sl}(3,\mathbb{H})$ if $\mathrm{sign}\, V=(4,0)$ or $(0,4)$,
\item[--] $\mathfrak{g}=\mathfrak{su}(3,3)$ or $\mathfrak{su}(2,4)$ if $\mathrm{sign}\, V=(3,1)$,
\item[--] $\mathfrak{g}=\mathfrak{su}(2,3)$ if $\mathrm{sign}\, V=(1,3)$,
\item[--] $\mathfrak{g}=\mathrm{F\,II}$ if $\mathrm{sign}\, V=(7,0)$,
\item[--] $\mathfrak{g}=\mathrm{F\,I}$ if $\mathrm{sign}\, V=(3,4)$,
\item[--] $\mathfrak{g}=\mathrm{E\,IV}$ if $\mathrm{sign}\, V=(8,0)$ or $(0,8)$,
\item[--] $\mathfrak{g}=\mathrm{E\,III}$ if $\mathrm{sign}\, V=(7,1)$,
\item[--] $\mathfrak{g}=\mathrm{E\,II}$ if $\mathrm{sign}\, V=(3,5)$,
\item[--] $\mathfrak{g}=\mathrm{E\,I}$ if $\mathrm{sign}\, V=(4,4)$,
\end{itemize}
with gradation as described in Theorem \ref{abcdefg}.\\
In all other dimensions and signatures, $\mathfrak{g}$ is not semi-simple, $\mathfrak{g}_0$ is as in Theorem \ref{lo0}, and $\mathfrak{g}_{i}=0$ for any $i\geq 1$.
\end{theorem}
\begin{proof}[Proof of Theorems \ref{N1c} and \ref{N1r}]
The constraints on the dimension and signature of $V$, for the existence of extended translation algebras with $N=1$, are proved in \cite{AC}.
For all dimensions and signatures that do not occur in Theorems \ref{theoC} and \ref{abcdefg}, the conclusion follows from Theorem \ref{lo0}, Theorem \ref{nonloso} and the results of section \ref{sec3}. 
\\ 
In all other cases, except the real cases of signature $(4,0)$, $(0,4)$ and $(3,1)$, 
it is straightforward to show that there is only one equivalence class of extended translation algebras with $N=1$.
Finally, similar, but much longer, elementary computations are necessary in the real cases with $\mathrm{sign} V=(4,0)$, $(0,4)$, $(3,1)$, where there are $4$ linearly independent admissible bilinear forms on $W$ whose invariants satisfy $\tau\sigma=-1$. In all cases one verifies that the maximal transitive prolongation is simple. The result then follows from Theorem \ref{theoC} and Theorem \ref{abcdefg}.
\end{proof}

\end{document}